\documentclass{amsart}
\usepackage{cleveref}
\usepackage{autonum}
\usepackage{mathtools}
\usepackage{amssymb}
\usepackage{enumerate}
\usepackage{braket}
\usepackage{array}
\usepackage{color}

\newtheorem{theorem}{Theorem}
\newtheorem{pretheorem}{Theorem}

\newtheorem{lemma}{Lemma}

\theoremstyle{definition}

\newtheorem{remark}{Remark}

\crefname{section}{Section}{Sections}

\crefname{theorem}{Theorem}{Theorems}
\crefname{pretheorem}{Theorem}{Theorems}
\crefname{corollary}{Corollary}{Corollaries}
\crefname{lemma}{Lemma}{Lemmas}
\crefname{definition}{Definition}{Definitions}
\crefname{remark}{Remark}{Remarks}
\crefname{example}{Example}{Examples}

\crefformat{equation}{\upshape(#2#1#3)}

\crefname{table}{Table}{Tables}

\newcommand{\midmid}{\mathrel{}\middle|\mathrel{}}

\renewcommand{\epsilon}{\varepsilon}

\newcommand{\GG}[1]{%
\frac{1}{k\ell}
\frac{\Gamma(\frac{#1}{k})\Gamma(\frac{1}{\ell})}
{\Gamma(\frac{#1}{k}+\frac{1}{\ell})}}
\newcommand{\GGG}[1]{%
\frac{1}{k\ell}
\frac{\Gamma(\frac{#1}{k})\Gamma(\frac{1}{\ell})}
{\Gamma(\frac{#1}{k}+\frac{1}{\ell}+1)}}

\def\showsize#1{\setbox0=\hbox{#1}width=\the\wd0, height=\the\ht0, depth=\the\dp0}

\begin{document}

\title[On the sum of a prime power and a power in short intervals]
{On the sum of a prime power and a power\\in short intervals}

\author[Y. Suzuki]{Yuta Suzuki}

\date{}


\begin{abstract}
Let $R_{k,\ell}(N)$ be the representation function
for the sum of the $k$-th power of a prime and the $\ell$-th power of a positive integer.
Languasco and Zaccagnini (2017) proved an asymptotic formula
for the average of $R_{1,2}(N)$ over short intervals $(X,X+H]$
of the length $H$ slightly shorter than $X^{\frac{1}{2}}$,
which is shorter than the length $H=X^{\frac{1}{2}+\epsilon}$
in the exceptional set estimates of Mikawa (1993) and of Perelli and Pintz (1995).
In this paper, we prove that the same asymptotic formula for $R_{1,2}(N)$ holds for $H$ of the size $X^{0.337}$.
Recently, Languasco and Zaccagnini (2018) extended their result to more general $(k,\ell)$.
We also consider this general case, and as a corollary,
we prove a conditional result of Languasco and Zaccagnini (2018)
for the case $\ell=2$ unconditionally up to some small factors.
\end{abstract}
\maketitle

%
%
\section{Introduction}
\label{section:intro}
Let $R(n)$ be the representation function
for a given additive problem with prime numbers.
For example, in this paper,
we consider the binary additive problem with prime numbers given by the equation
\begin{equation}
\label{main_eq}
N=p^k+n^\ell,
\end{equation}
where $k,\ell$ are given positive integers,
$p$ denotes a variable for prime numbers,
and $n$ denotes a variable for positive integers.
Then the representation function for the equation \cref{main_eq}
with logarithmic weight is given by
\begin{equation}
\label{def:R}
R(N)=R_{k,\ell}(N)=\sum_{p^k+n^\ell=N}\log p,
\end{equation}
which counts the solutions $(p,n)$ of \cref{main_eq}.
In this paper, we consider the short interval average of such representation function
\begin{equation}
\label{short_avg}
\sum_{X<N\le X+H}R(N),
\end{equation}
where $4\le H\le X$.
Recently, Languasco and Zaccagnini gave extensive research
(e.g.~see \cite{LZ_3/2,LZ_1,LZ_4444,LZ_122,LZ_kl,LZ_kl2}) on the short interval average \eqref{short_avg}
for various additive problems with prime numbers,
and in the case $k=1$ of \cref{main_eq},
they obtained short interval asymptotic formulas for the average \cref{short_avg}
with $H$ shorter than in the known exceptional set estimates in short intervals.

For example, let us consider the Hardy--Littlewood equation
\begin{equation}
\label{HL_eq}
N=p+n^2,
\end{equation}
which is the case $(k,\ell)=(1,2)$ of our equation \cref{main_eq}.
In their famous paper Partitio Numerorum III,
Hardy and Littlewood~\cite[Conjecture~H]{PN3} applied their circle method formally
to obtain a hypothetical asymptotic formula
\begin{equation}
\label{HL_asymp}
R_{1,2}(N)=\mathfrak{S}(N)\sqrt{N}+(\text{error}),\quad
(N\colon\text{not square})
\end{equation}
as $N\to\infty$, where the singular series $\mathfrak{S}(N)$ is given by
\[
\mathfrak{S}(N)=\prod_{p>2}\left(1-\frac{(N/p)}{p-1}\right),\quad
(N/p)\colon\text{Legendre symbol}.
\]
This asymptotic formula \cref{HL_asymp} itself
seems still far beyond our current technology,
but we can prove \cref{HL_asymp} on average.
Let $A>0$ be an arbitrary constant and introduce
\[
E(X)
=\#\left\{N\le X\midmid
\left|R_{1,2}(N)-\mathfrak{S}(N)\sqrt{N}\right|\ge\sqrt{N}(\log N)^{-A},\ 
N\colon\text{not square}\right\},
\]
where $X\ge2$ is a real number.
This function $E(X)$ counts the number of positive integers $\le X$
for which the hypothetical asymptotic formula \eqref{HL_asymp} fails.
Miech~\cite{Miech} proved a non-trivial bound
\begin{equation}
\label{Miech}
E(X)\ll XL^{-A},\quad
L=\log X
\end{equation}
for any $A>0$,
where the implicit constant depends on $A$.
Thus, Miech proved that the asymptotic formula \cref{HL_asymp}
holds for almost all integer $N$.
The short interval version of Miech's result~\cref{Miech}
was obtained by Mikawa~\cite{Mikawa} and by Perelli and Pintz \cite{PP_square}
independently. Their result gives a non-trivial bound
\begin{equation}
\label{MPP_bound}
E(X+H)-E(X)\ll HL^{-A}
\end{equation}
for any $A>0$ provided
\begin{equation}
\label{MPP_range}
X^{\frac{1}{2}+\epsilon}\le H\le X,
\end{equation}
where $X,H,\epsilon$ are real numbers with $4\le H\le X$ and $\epsilon>0$,
and the implicit constant may depend on $A$ and $\epsilon$.
One of the aim in this problem is to obtain the same bound \cref{MPP_bound} for shorter $H$.
Although the range~\cref{MPP_range} is still the best possible result today
for the estimate \cref{MPP_bound},
Languasco and Zaccagnini~\cite{LZ_3/2} showed
that if we consider the direct average~\cref{short_avg} instead,
then we can deal with shorter $H$ than \cref{MPP_range}.
After some minor modification,
Theorem 2 of \cite{LZ_3/2} gives the following.
In this paper, the letter $B$ denotes the quantity given by
\begin{equation}
\label{def:B}
B=\exp\left(c\left(\frac{\log X}{\log\log X}\right)^{\frac{1}{3}}\right),
\end{equation}
where $c$ is some small positive constant
which may depend on $k,\ell$ and $\epsilon$.

\begin{pretheorem}[{Languasco and Zaccagnini~\cite[Theorem 2]{LZ_3/2}}]
\label{thm:LZ_12}
For real numbers $X,H$ and $\epsilon$ with $4\le H\le X$ and $\epsilon>0$,
we have
\begin{equation}
\label{LZ_asymp_12}
\sum_{X<N\le X+H}R_{1,2}(N)
=
HX^{\frac{1}{2}}+O(HX^{\frac{1}{2}}B^{-1})
\end{equation}
provided $X^{\frac{1}{2}}B^{-1}\le H\le X^{1-\epsilon}$,
where the implicit constant depends on $\epsilon$.
\end{pretheorem}

Thus, Languasco and Zaccagnini obtained the asymptotic formula \cref{LZ_asymp_12}
for $H$ shorter than \cref{MPP_range} up to the factor $B^{-1}$.
However, we still have the same exponent $\frac{1}{2}$ of $X$.
In this paper, we improve this exponent from $\frac{1}{2}$ to $0.336899\cdots$.
\begin{theorem}
\label{thm:main_12}
For real numbers $X,H,\epsilon$ with $4\le H\le X$ and $\epsilon>0$,
we have the asymptotic formula \cref{LZ_asymp_12}
provided
\[
X^{\Theta(1,2)+\epsilon}\le H\le X^{1-\epsilon},\quad
\Theta(1,2)=\frac{32-4\sqrt{15}}{49}=0.336899\cdots,
\]
where the implicit constant depends on $\epsilon$.
\end{theorem}

Recently, Languasco and Zaccagnini~\cite{LZ_kl,LZ_kl2} dealt with
other cases of \cref{main_eq}:
\begin{pretheorem}[{Languasco and Zaccagnini~\cite[Theorem~3]{LZ_kl2}}]
\label{thm:LZ_kl}
For positive integers $k,\ell$ with $k,\ell\ge2$,
and real numbers $X,H,\epsilon$ with $4\le H\le X$ and $\epsilon>0$, we have
\begin{equation}
\label{LZ_asymp_kl}
\sum_{X<N\le X+H}R_{k,\ell}(N)
=
\GG{1}HX^{\frac{1}{k}+\frac{1}{\ell}-1}+O(HX^{\frac{1}{k}+\frac{1}{\ell}-1}B^{-1})
\end{equation}
provided
\[
X^{\Theta_{LZ}(k,\ell)+\epsilon}\le H\le X^{1-\epsilon},
\]
where
\[
\Theta_{LZ}(k,\ell)
=
1-\theta_{\mathrm{LZ}}(k,\ell),\quad
\theta_{\mathrm{LZ}}(k,\ell)=\min\left(\frac{5}{6k},\frac{1}{\ell}\right),
\]
and the implicit constant depends on $k,\ell$ and $\epsilon$.
\end{pretheorem}

\begin{remark}
In \cite{LZ_kl2}, Languasco and Zaccagnini considered
\[
\tilde{R}_{k,\ell}(N)
=
\sum_{\substack{p^k+n^\ell=N\\N/B<p^k,n^{\ell}\le N}}\log p
\]
instead of \cref{def:R}. However, new restrictions
\[
N/B<p^k,n^{\ell}\le N
\]
are introduced just for some technical simplicity of the proof.
Indeed, it is easy to replace $\tilde{R}_{k,\ell}(N)$ by $R_{k,\ell}(N)$
assuming
$X^{1-\min(\frac{1}{k},\frac{1}{\ell})}\le H\le X$
as follows. If we remove the restriction $p^k>N/B$, then the resulting error is bounded by
\[
\ll
\sum_{\substack{X<p^k+m^\ell\le X+H\\ p^{k}\le 2X/B}}\log p
\ll
L\sum_{p^{k}\le 2X/B}\sum_{X-p^k<m^{\ell}\le X+H-p^k}1.
\]
By \cref{lem:power_short} below and the assuminption $H\ge X^{1-\frac{1}{\ell}}$, this is
\[
\ll
HL\sum_{p^{k}\le 2X/B}(X-p^k)^{\frac{1}{\ell}-1}
\ll
HX^{\frac{1}{\ell}-1}(X/B)^{\frac{1}{k}}
\ll
HX^{\frac{1}{k}+\frac{1}{\ell}-1}B^{-\frac{1}{k}},
\]
which is bounded by the error term of \cref{thm:LZ_kl} up to replacing the constant $c$
in \cref{def:B}. The restriction $n^k>N/B$ can be removed in the same way.
\end{remark}

Actually, \cref{thm:main_12} above is a special case of the following general result:
\begin{theorem}
\label{thm:main_kl}
For positive integers $k,\ell$ with $\ell\ge2$, and
real numbers $X,H,\epsilon$ with $4\le H\le X$ and $\epsilon>0$,
we have the asymptotic formula \cref{LZ_asymp_kl}
provided
\[
X^{\Theta(k,\ell)+\epsilon}\le H\le X^{1-\epsilon},
\]
where $\Theta(k,\ell)$ is defined by
\begin{equation}
\label{def:Theta}
\begin{aligned}
\Theta(k,\ell)
&=
1-\theta(k,\ell),\quad
\theta(k,\ell)=
\max(\theta_{A}(k,\ell),\theta_{B}(k,\ell)),\\
\theta_{A}(k,\ell)
&=
\min\left(\frac{\lambda_1(\ell)}{k},\frac{\lambda_2(k,\ell)}{k},\frac{k}{\ell(k-1)}\right),\\
\theta_{B}(k,\ell)
&=
\min\left(\frac{5}{12k},\frac{k}{\ell(k-1)}\right),\\
\lambda_1(\ell)
&=
\left\{
\begin{array}{>{\displaystyle}ll}
\frac{\ell}{2(\ell-1)}
&(\text{if $2\le\ell\le3$}),\\[3mm]
\frac{3\ell^2+2\sqrt{3}\ell^{\frac{3}{2}}+\ell}{(3\ell-1)^2}
&(\text{if $3\le\ell\le\frac{25}{3}$}),\\[3mm]
\frac{5\ell}{4(3\ell-5)}
&(\text{if $\ell\ge\frac{25}{3}$}),
\end{array}
\right.\\[2mm]
\lambda_2(k,\ell)
&=
\left\{
\begin{array}{>{\displaystyle}ll}
\frac{2}{3}\left(\frac{k}{\ell}+\frac{1}{2}\right)
&(\text{if $\frac{5}{8}\ell\le k$}),\\[3mm]
\frac{10}{49}+\frac{2k}{7\ell}+\frac{4}{7}\sqrt{\frac{6}{7}\left(\frac{k}{\ell}-\frac{1}{7}\right)}
&(\text{if $\frac{31}{96}\ell\le k\le\frac{5}{8}\ell$}),\\[3mm]
\frac{10}{11}\left(\frac{k}{\ell}+\frac{1}{4}\right)
&(\text{if $k\le\frac{31}{96}\ell$}),
\end{array}
\right.
\end{aligned}
\end{equation}
and the implicit constant depends on $k,\ell$ and $\epsilon$.
\end{theorem}

We prove \cref{thm:main_kl} at the end of \cref{section:proof}.

The mainly concerned case of \cref{thm:main_kl} is the case
\begin{equation}
\label{maincase}
\theta_{A}(k,\ell)>\theta_{B}(k,\ell),\theta_{LZ}(k,\ell).
\end{equation}
We compare these three exponents in \cref{section:comparison}.
It turns out that \cref{maincase} occurs for
\begin{equation}
\label{comparison}
\begin{aligned}
\ell=2,\quad&\text{or}\\
3\le\ell\le 9\quad&\text{and}\quad \frac{5}{24}\ell<k<\lambda_1(\ell)\ell,\quad\text{or}\\
\ell\ge10\quad&\text{and}\quad\frac{5}{24}\ell+\frac{1}{24}\sqrt{\ell(25\ell-240)}<k<\lambda_1(\ell)\ell.\\
\end{aligned}
\end{equation}
Furthermore, in \cref{section:comparison}, we also see that
\begin{equation}
\label{theta_determined}
\theta(k,\ell)
=
\left\{
\begin{array}{>{\displaystyle}cl}
\frac{\lambda_2(k,\ell)}{k}&(\text{for $(k,\ell)=(1,2),(1,3),(1,4)$}\\
&\hspace{8mm}\text{$(2,5),(2,6),(2,7),(2,8),(2,9)$}),\\[2mm]
\theta_{B}(k,\ell)&(\text{for $k=1$ and $\ell\ge5$}),\\[2mm]
\min\left(\frac{\lambda_1(\ell)}{k},\frac{k}{\ell(k-1)}\right)&(\text{otherwise}).
\end{array}
\right.
\end{equation}
In \cref{Table:Theta},
we list up which exponent gives the best result in the range $1\le k\le 10$ and $2\le\ell\le20$.
The case \cref{maincase} occurs for at least one $k$ for each $\ell$ since
\[
\frac{5}{24}\ell+\frac{1}{24}\sqrt{\ell(25\ell-240)}
=
\frac{5}{24}\ell+\frac{5}{24}\ell\sqrt{1-\frac{48}{5\ell}}
\le
\frac{5}{12}\ell-1
\quad\text{and}\quad
\lambda_1(\ell)\ell>\frac{5}{12}\ell.
\]
However, it happens only in the small neighborhood of the line $k=\frac{5}{12}\ell$ for $\ell\ge3$.
In contrast, for $\ell=2$, \cref{maincase} is always the case.
In particular, as a corollary of \cref{thm:main_kl}, we can obtain the exponent $1-\frac{1}{k}$
unconditionally for the case $k\ge2$ and $\ell=2$, which was obtained under the Riemann hypothesis
by Languasco and Zaccagnini~\cite[Theorem~1.4]{LZ_kl}:
\begin{theorem}
\label{thm:main_k2}
For positive integer $k$ with $k\ge2$,
and real numbers $X,H,\epsilon$
with $4\le H\le X$ and $\epsilon>0$,
we have the asymptotic formula \cref{LZ_asymp_kl} with $\ell=2$
provided
\[
X^{1-\frac{1}{k}+\epsilon}\le H\le X^{1-\epsilon},
\]
where the implicit constant depends on $k$ and $\epsilon$.
\end{theorem}
We prove \cref{thm:main_k2} at the end of \cref{section:proof}.

\begin{table}[htb]
\centering
\setlength{\tabcolsep}{3pt}
\caption{The best exponents in $\theta_A,\theta_B,\theta_{LZ}$}
\label{Table:Theta}
\newcolumntype{C}[1]{>{\centering\arraybackslash}m{#1}}
\begin{tabular}{c|*{10}{C{4.8mm}}}
$\ell\mathrel{\backslash}k$&1&2&3&4&5&6&7&8&9&10\\
\hline
2&$A$&$A$&$A$&$A$&$A$&$A$&$A$&$A$&$A$&$A$\\
3&$A$&$A$&$LZ$&$LZ$&$LZ$&$LZ$&$LZ$&$LZ$&$LZ$&$LZ$\\
4&$A$&$A$&$LZ$&$LZ$&$LZ$&$LZ$&$LZ$&$LZ$&$LZ$&$LZ$\\
5&$B$&$A$&$A$&$LZ$&$LZ$&$LZ$&$LZ$&$LZ$&$LZ$&$LZ$\\
6&$B$&$A$&$A$&$LZ$&$LZ$&$LZ$&$LZ$&$LZ$&$LZ$&$LZ$\\
7&$B$&$A$&$A$&$LZ$&$LZ$&$LZ$&$LZ$&$LZ$&$LZ$&$LZ$\\
8&$B$&$A$&$A$&$A$&$LZ$&$LZ$&$LZ$&$LZ$&$LZ$&$LZ$\\
9&$B$&$A$&$A$&$A$&$LZ$&$LZ$&$LZ$&$LZ$&$LZ$&$LZ$\\
10&$B$&$B$&$A$&$A$&$LZ$&$LZ$&$LZ$&$LZ$&$LZ$&$LZ$\\
11&$B$&$B$&$B$&$A$&$A$&$LZ$&$LZ$&$LZ$&$LZ$&$LZ$\\
12&$B$&$B$&$B$&$A$&$A$&$LZ$&$LZ$&$LZ$&$LZ$&$LZ$\\
13&$B$&$B$&$B$&$B$&$A$&$A$&$LZ$&$LZ$&$LZ$&$LZ$\\
14&$B$&$B$&$B$&$B$&$A$&$A$&$LZ$&$LZ$&$LZ$&$LZ$\\
15&$B$&$B$&$B$&$B$&$B$&$A$&$A$&$LZ$&$LZ$&$LZ$\\
16&$B$&$B$&$B$&$B$&$B$&$A$&$A$&$LZ$&$LZ$&$LZ$\\
17&$B$&$B$&$B$&$B$&$B$&$A$&$A$&$LZ$&$LZ$&$LZ$\\
18&$B$&$B$&$B$&$B$&$B$&$B$&$A$&$A$&$LZ$&$LZ$\\
19&$B$&$B$&$B$&$B$&$B$&$B$&$A$&$A$&$LZ$&$LZ$\\
20&$B$&$B$&$B$&$B$&$B$&$B$&$B$&$A$&$A$&$LZ$\\
\end{tabular}
\end{table}

Languasco and Zaccagnini applied the circle method
to prove \cref{thm:LZ_12} and \cref{thm:LZ_kl}.
In this paper, we deal with the average~\cref{short_avg} rather more directly.
Our argument is similar to the classical proof of the prime number theorem in short intervals.
We first insert the von Mangoldt explicit formula.
Then, we apply the Poisson summation formula in order
to detect the cancellations over the sequence $n^{\ell}$,
which is not involved in the proof of Languasco and Zaccagnini.
Finally, we estimate the sum over non-trivial zeros
of the Riemann zeta function by using the Huxley--Ingham zero density estimate.

\section{Notations and conventions}
We use the following notations and conventions.

As usual, let $\Lambda(n)$ be the von Mangoldt function and
\begin{equation}
\label{def:psi}
\psi(x)=\sum_{m\le x}\Lambda(m).
\end{equation}
We denote the Riemann zeta function by $\zeta(s)$.
By $\rho=\beta+i\gamma$, we denote non-trivial zeros
of $\zeta(s)$ with the real part $\beta$ and the imaginary part $\gamma$.
For a real number $\alpha$ and $T$ with $T\ge0$,
let $N(\alpha,T)$ be the number of non-trivial zeros $\rho=\beta+i\gamma$
of $\zeta(s)$ in the rectangle $\alpha\le\beta\le1$ and $|\gamma|\le T$ counted with multiplicity.

For a complex valued function $f$ defined over an interval $[a,b]$,
let $V_{[a,b]}(f)$ be the total variation of $f$ over $[a,b]$,
and
\begin{equation}
\label{def:norm}
\|f\|=\|f\|_{BV([a,b])}=\sup_{x\in[a,b]}|f(x)|+V_{[a,b]}(f).
\end{equation}
For a real number $x$, let $e(x)=\exp(2\pi i x)$
$[x]$ be the largest integer not exceeding $x$,
and $\{x\}=x-[x]$.

The letters $X,H,Q$ denote real numbers,
and they are always assumed to satisfy
\begin{equation}
\label{Q_range}
4\le H\le X,\quad
X\le Q\le X+H.
\end{equation}
The letters $c_0, c_1>0$ denote some small absolute constants
and $c$ denotes a constant with $0<c\le 1$ which may depend on $k,\ell$ and $\epsilon$.
The letters $B$ and $L$ are used for the quantities
\[
B=\exp\left(c\left(\frac{\log X}{\log\log X}\right)^{\frac{1}{3}}\right),\quad
L=\log X.
\]
For positive integers $k,\ell$ and a non-zero complex number $\alpha$, we let
\begin{equation}
\label{def:S}
S_{\alpha}(Q)
=S_{\alpha,k,\ell}(Q;X)
=\frac{1}{\alpha}\sum_{n^{\ell}\le X}(Q-n^{\ell})^{\frac{\alpha}{k}},\quad
S(Q)=S_{1}(Q).
\end{equation}

Let $\phi(\lambda)$ be a function defined over $[0,+\infty)$ by
\begin{equation}
\label{def:phi}
\phi(\lambda)
=
\left\{
\begin{array}{>{\displaystyle}ll}
\frac{3}{5}\lambda+\frac{3}{4}
&(\text{if $0\le\lambda\le\frac{25}{48}$}),\\[3mm]
3\lambda+2(1-\sqrt{3\lambda})&(\text{if $\frac{25}{48}\le\lambda\le\frac{3}{4}$}),\\[3mm]
\lambda+\frac{1}{2}&(\text{if $\frac{3}{4}\le\lambda\le1$}).
\end{array}
\right.
\end{equation}
This function will be used for estimating sums over non-trivial zeros of $\zeta(s)$.
For real numbers $k,\ell$ with $k\ge1$ and $\ell\ge2$,
we also introduce two real-valued functions $\lambda_1(\ell)$ and $\lambda_2(k,\ell)$
as in \cref{thm:main_kl} by
\begin{equation}
\label{def:lambda}
\begin{aligned}
\lambda_1(\ell)
&=
\left\{
\begin{array}{>{\displaystyle}ll}
\frac{\ell}{2(\ell-1)}
&(\text{if $2\le\ell\le3$}),\\[3mm]
\frac{3\ell^2+2\sqrt{3}\ell^{\frac{3}{2}}+\ell}{(3\ell-1)^2}
&(\text{if $3\le\ell\le\frac{25}{3}$}),\\[3mm]
\frac{5\ell}{4(3\ell-5)}
&(\text{if $\ell\ge\frac{25}{3}$}),
\end{array}
\right.\\[3mm]
\lambda_2(k,\ell)
&=
\left\{
\begin{array}{>{\displaystyle}ll}
\frac{2}{3}\left(\frac{k}{\ell}+\frac{1}{2}\right)
&(\text{if $\frac{5}{8}\ell\le k$}),\\[3mm]
\frac{10}{49}+\frac{2k}{7\ell}+\frac{4}{7}\sqrt{\frac{6}{7}\left(\frac{k}{\ell}-\frac{1}{7}\right)}
&(\text{if $\frac{31}{96}\ell\le k\le\frac{5}{8}\ell$}),\\[3mm]
\frac{10}{11}\left(\frac{k}{\ell}+\frac{1}{4}\right)
&(\text{if $k\le\frac{31}{96}\ell$}).
\end{array}
\right.
\end{aligned}
\end{equation}
These functions are used in the exponent of the admissible ranges for $X$ and $H$.

We have several expressions of the form
\[
\min\left(A,\infty\right).
\]
As a convention, we define this quantity by
\[
\min\left(A,\infty\right)
=
A.
\]

If Theorem or Lemma is stated
with the phrase ``where the implicit constant depends on $a,b,c,\ldots$'',
then every implicit constant in the corresponding proof
may also depend on $a,b,c,\ldots$ even without special mentions.

\section{Preliminary lemmas}
In this section, we prepare some lemmas for the proof of \cref{thm:main_kl}.
We start with some simple estimates for short interval sums without prime numbers.

\begin{lemma}
\label{lem:power_short}
For positive integer $\ell$ and real numbers $X,H$ with $X,H\ge2$,
\[
\sum_{X<n^{\ell}\le X+H}1\ll HX^{\frac{1}{\ell}-1}+1,
\]
where the implicit constant is absolute.
\end{lemma}
\begin{proof}
By using $x-1<[x]\le x$, we see that
\begin{align}
\sum_{X<n^{\ell}\le X+H}1
&=
[(X+H)^{\frac{1}{\ell}}]-[X^{\frac{1}{\ell}}]
\le
(X+H)^{\frac{1}{\ell}}-X^{\frac{1}{\ell}}+1\\
&=
\frac{1}{\ell}\int_{X}^{X+H}u^{\frac{1}{\ell}-1}du+1
\ll
HX^{\frac{1}{\ell}-1}+1.
\end{align}
This completes the proof.
\end{proof}

\begin{lemma}
\label{lem:main_term_diff}
For positive integers $k,\ell$ and real numbers $X,H$ with $4\le H\le X$,
\begin{align}
&\GGG{1}\left((X+H)^{\frac{1}{k}+\frac{1}{\ell}}-X^{\frac{1}{k}+\frac{1}{\ell}}\right)\\
&=
\GG{1}HX^{\frac{1}{k}+\frac{1}{\ell}-1}
+
O(H^2X^{\frac{1}{k}+\frac{1}{\ell}-2}),
\end{align}
where the implicit constant is absolute.
\end{lemma}
\begin{proof}
By the fundamental theorem of calculus,
\begin{equation}
\label{main_term_diff:int}
\begin{aligned}
&\GGG{1}\left((X+H)^{\frac{1}{k}+\frac{1}{\ell}}-X^{\frac{1}{k}+\frac{1}{\ell}}\right)\\
&=
\GG{1}\int_{X}^{X+H}u^{\frac{1}{k}+\frac{1}{\ell}-1}du.
\end{aligned}
\end{equation}
For $X<u\le X+H$,
by using the mean value theorem, we have
\[
u^{\frac{1}{k}+\frac{1}{\ell}-1}
=
X^{\frac{1}{k}+\frac{1}{\ell}-1}
+
O\left(HX^{\frac{1}{k}+\frac{1}{\ell}-2}\right).
\]
Thus, the integral in \cref{main_term_diff:int} can be rewritten as
\[
\int_{X}^{X+H}u^{\frac{1}{k}+\frac{1}{\ell}-1}du
=
HX^{\frac{1}{k}+\frac{1}{\ell}-1}
+
O\left(H^2X^{\frac{1}{k}+\frac{1}{\ell}-2}\right).
\]
On inserting this formula into  \cref{main_term_diff:int},
and noting that
\[
\GG{1}
\ll
\frac{1}{k\ell}\frac{(\frac{1}{k})^{-1}(\frac{1}{\ell})^{-1}}{(\frac{1}{k}+\frac{1}{\ell})^{-1}}
\ll
\frac{1}{k}+\frac{1}{\ell}
\ll1,
\]
we arrive at the lemma. 
\end{proof}

\begin{lemma}
\label{lem:main_term}
For positive integers $k,\ell$
and real numbers $X,H$ with $4\le H\le X$,
\begin{align}
S(X+H)-S(X)
=
\GG{1} HX^{\frac{1}{k}+\frac{1}{\ell}-1}
+
O\left(H^{1+\frac{1}{k}}X^{\frac{1}{\ell}-1}+H^{\frac{1}{k}}\right),
\end{align}
where $S(Q)$ is defined by
\[
S(Q)=\sum_{n^{\ell}\le X}\left(Q-n^{\ell}\right)^{\frac{1}{k}}
\]
as in \cref{def:S} and the implicit constant is absolute.
\end{lemma}
\begin{proof}
The left-hand side of the assertion is
\begin{equation}
\label{main_term:first}
\begin{aligned}
=
\frac{1}{k}\sum_{n^{\ell}\le X}\int_{X}^{X+H}(u-n^{\ell})^{\frac{1}{k}-1}du
=
\int_{X}^{X+H}\frac{1}{k}\sum_{n^{\ell}\le X}(u-n^{\ell})^{\frac{1}{k}-1}du.
\end{aligned}
\end{equation}
Since the function $(u-w^{\ell})^{\frac{1}{k}-1}$ is non-decreasing
over $0\le w\le X^{\frac{1}{\ell}}$,
\begin{equation}
\label{main_term:integrand}
\begin{aligned}
\frac{1}{k}\sum_{n^{\ell}\le X}(u-n^{\ell})^{\frac{1}{k}-1}
&=
\frac{1}{k}\int_{0}^{X^{\frac{1}{\ell}}}(u-w^{\ell})^{\frac{1}{k}-1}dw
+
O\left(\frac{1}{k}(u-X)^{\frac{1}{k}-1}\right)\\
&=
\frac{1}{k\ell}\int_{0}^{X}(u-w)^{\frac{1}{k}-1}w^{\frac{1}{\ell}-1}dw
+
O\left(\frac{1}{k}(u-X)^{\frac{1}{k}-1}\right)
\end{aligned}
\end{equation}
for $X<u\le X+H$. Note that the second term on the right-hand side may tend to $\infty$
as $u\to X+0$, but this term is integrable over $(X,X+H]$.
We next extend the integral on the right-hand side.
For $X<u\le X+H$, by changing the variable,
\begin{equation}
\begin{aligned}
\frac{1}{k\ell}\int_{X}^{u}(u-w)^{\frac{1}{k}-1}w^{\frac{1}{\ell}-1}dw
&\ll
\frac{1}{k\ell}X^{\frac{1}{\ell}-1}\int_{0}^{u-X}w^{\frac{1}{k}-1}dw\\
&\ll
\frac{1}{\ell}X^{\frac{1}{\ell}-1}(u-X)^{\frac{1}{k}}
\ll
\frac{1}{\ell}H^{\frac{1}{k}}X^{\frac{1}{\ell}-1}.
\end{aligned}
\end{equation}
Hence, we can extend the integral in \cref{main_term:integrand} as
\begin{equation}
\begin{aligned}
&\frac{1}{k}\sum_{n^{\ell}\le X}(u-n^{\ell})^{\frac{1}{k}-1}\\
&=
\frac{1}{k\ell}\int_{0}^{u}(u-w)^{\frac{1}{k}-1}w^{\frac{1}{\ell}-1}dw
+
O\left(\frac{1}{k}(u-X)^{\frac{1}{k}-1}+H^{\frac{1}{k}}X^{\frac{1}{\ell}-1}\right).
\end{aligned}
\end{equation}
The last integral on the right-hand side is
\begin{align}
\frac{1}{k\ell}\int_{0}^{u}(u-w)^{\frac{1}{k}-1}w^{\frac{1}{\ell}-1}dw
&=
\frac{1}{k\ell}u^{\frac{1}{k}+\frac{1}{\ell}-1}
\int_{0}^{1}(1-w)^{\frac{1}{k}-1}w^{\frac{1}{\ell}-1}dw\\
&=
\GG{1}u^{\frac{1}{k}+\frac{1}{\ell}-1}.
\end{align}
Therefore,
\[
\frac{1}{k}\sum_{n^{\ell}\le X}(u-n^{\ell})^{\frac{1}{k}-1}
=
\GG{1} u^{\frac{1}{k}+\frac{1}{\ell}-1}
+
O\left(\frac{1}{k}(u-X)^{\frac{1}{k}-1}+H^{\frac{1}{k}}X^{\frac{1}{\ell}-1}\right).
\]
On inserting this formula into \cref{main_term:first},
the left-hand side of the assertion is
\begin{equation}
\label{main_term:pre}
=
\GG{1}\int_{X}^{X+H}u^{\frac{1}{k}+\frac{1}{\ell}-1}du
+O\left(H^{\frac{1}{k}}+H^{1+\frac{1}{k}}X^{\frac{1}{\ell}-1}\right).
\end{equation}
By \cref{lem:main_term_diff}, this is
\[
=
\GG{1}HX^{\frac{1}{k}+\frac{1}{\ell}-1}
+O\left(H^{2}X^{\frac{1}{k}+\frac{1}{\ell}-2}
+H^{1+\frac{1}{k}}X^{\frac{1}{\ell}-1}+H^{\frac{1}{k}}\right).
\]
Since $H\le X$, we can estimate the first error term as
\[
H^{2}X^{\frac{1}{k}+\frac{1}{\ell}-2}
=
H^{2}X^{-(1-\frac{1}{k})}X^{\frac{1}{\ell}-1}
\le
H^{1+\frac{1}{k}}X^{\frac{1}{\ell}-1}.
\]
This completes the proof.
\end{proof}

\begin{lemma}
\label{lem:beta_short}
For positive integers $k,\ell$ and real numbers $X,H$ with $4\le H\le X$,
\[
\sum_{X<m^k+n^{\ell}\le X+H}1
\ll
HX^{\frac{1}{k}+\frac{1}{\ell}-1}
+
H^{\frac{1}{k}}
+
X^{\frac{1}{\ell}},
\]
where the implicit constant is absolute.
\end{lemma}
\begin{proof}
We rewrite the left-hand side as
\[
\sum_{X<m^k+n^{\ell}\le X+H}1
=
\sum_{n^{\ell}\le X+H}\sum_{X-n^{\ell}<m^k\le X+H-n^{\ell}}1.
\]
We next truncate the outer summation over $n^{\ell}$.
By using \cref{lem:power_short},
\[
\sum_{X<n^{\ell}\le X+H}\sum_{X-n^{\ell}<m^k\le X+H-n^{\ell}}1
\ll
\sum_{X<n^{\ell}\le X+H}\sum_{m^k\le H}1
\ll
H^{1+\frac{1}{k}}X^{\frac{1}{\ell}-1}+H^{\frac{1}{k}}.
\]
Thus, by using the assumption $H\le X$,
\begin{equation}
\label{beta_short:pre}
\sum_{X<m^k+n^{\ell}\le X+H}1
=
\sum_{n^{\ell}\le X}\sum_{X-n^{\ell}<m^k\le X+H-n^{\ell}}1
+
O\left(HX^{\frac{1}{k}+\frac{1}{\ell}-1}
+
H^{\frac{1}{k}}\right).
\end{equation}
The sum on the right-hand side is
\[
\sum_{n^{\ell}\le X}\sum_{X-n^{\ell}<m^k\le X+H-n^{\ell}}1
=
S(X+H)-S(X)
+
O(X^{\frac{1}{\ell}}).
\]
By using \cref{lem:main_term} and the assumption $H\le X$,
\[
\sum_{n^{\ell}\le X}\sum_{X-n^{\ell}<m^k\le X+H-n^{\ell}}1
\ll
HX^{\frac{1}{k}+\frac{1}{\ell}-1}+H^{\frac{1}{k}}+X^{\frac{1}{\ell}}.
\]
On inserting this estimate into \cref{beta_short:pre},
we obtain the lemma.
\end{proof}

We next recall some standard lemmas on prime numbers
and non-trivial zeros of the Riemann zeta functions.

\begin{lemma}
\label{lem:vonMangoldt}
For real numbers $X,T,x$ with $2\le T\le 2X$ and $0\le x\le X$, we have
\[
\psi(x)
=
x-\sum_{\substack{\rho\\|\gamma|\le T}}\frac{x^{\rho}}{\rho}+O(XT^{-1}L^2),
\]
where the implicit constant is absolute.
\end{lemma}
\begin{proof}
In the case $2\le x\le X$,
this follows from Theorem~12.5 of \cite{MV}.
In the case $0\le x\le 2$, the lemma trivially follows
since $XT^{-1}L^2\gg L^2$ by $T\le 2X$, and
\[
\sum_{\substack{\rho\\|\gamma|\le T}}\frac{x^{\rho}}{\rho}
\ll
\sum_{\substack{\rho\\|\gamma|\le T}}\frac{1}{|\rho|}
\ll
(\log T)^2
\ll
L^2
\]
for the case $0\le x\le2$. This completes the proof.
\end{proof}

\begin{lemma}[The Korobov--Vinogradov zero-free region]
\label{lem:KV}
We have $\zeta(s)\neq0$ for
\[
\sigma>1-c_0(\log\tau)^{-\frac{2}{3}}(\log\log\tau)^{-\frac{1}{3}},\quad
s=\sigma+it,\quad
\tau=|t|+4,
\]
where $c_0>0$ is some absolute constant.
\end{lemma}
\begin{proof}
See \cite[Theorem 6.1, p.~143]{Ivic}.
Note that by taking $c_0>0$ sufficiently small,
we can remove the condition $t\ge t_0$ in Theorem~6.1 of \cite{Ivic}.
\end{proof}

\begin{lemma}[The Huxley--Ingham zero density estimate
{\cite{Huxley,Ingham}}]
\label{lem:HI}
For real numbers $\alpha$ and $T$ with $\frac{1}{2}\le\alpha\le1$ and $T\ge2$,
\[
N(\alpha,T)\ll T^{c(\alpha)}(\log T)^A,\quad
c(\alpha)=
\left\{
\begin{array}{ll}
\frac{3(1-\alpha)}{3\alpha-1}&(\text{if $\frac{3}{4}\le\alpha\le1$}),\\[3mm]
\frac{3(1-\alpha)}{2-\alpha}&(\text{if $\frac{1}{2}\le\alpha\le\frac{3}{4}$}),
\end{array}
\right.
\]
where the constant $A$ and the implicit constant are absolute.
\end{lemma}
\begin{proof}
See \cite[Theorem 11.1, p.~273]{Ivic}.
\end{proof}

\begin{lemma}
\label{lem:Huxley_PNT}
For real numbers $X,H,\epsilon$ with $4\le H\le X$ and $\epsilon>0$,
\[
\psi(X+H)-\psi(X)=H+O(HB^{-1})
\]
provided
\[
X^{\frac{7}{12}+\epsilon}\le H\le X,
\]
where the implicit constant depends on $\epsilon$.
\end{lemma}
\begin{proof}
This follows by \cref{lem:KV} and \cref{lem:HI} through the standard argument.
\end{proof}

In the proof of \cref{thm:main_kl},
we need to estimate several sums over non-trivial zeros
of the Riemann zeta function. Our next several lemmas
deal with such sums and the exponents in the resulting estimates.

\begin{lemma}
\label{lem:zero_sum}
For real numbers $K,X,Y$ with $1\le K\le Y\le X^2$ and $X\ge4$,
\[
\sum_{\substack{\rho\\K<|\gamma|\le 2K}}Y^{\beta}
\ll
\left(Y^{\phi(\lambda)}+Y^{1-\eta+2\eta\lambda}\right)L^A,
\]
where the function $\phi(\lambda)$ is defined by
\begin{equation}
\phi(\lambda)
=
\left\{
\begin{array}{>{\displaystyle}ll}
\frac{3}{5}\lambda+\frac{3}{4}
&(\text{if $0\le\lambda\le\frac{25}{48}$}),\\[3mm]
3\lambda+2(1-\sqrt{3\lambda})&(\text{if $\frac{25}{48}\le\lambda\le\frac{3}{4}$}),\\[3mm]
\lambda+\frac{1}{2}&(\text{if $\frac{3}{4}\le\lambda\le1$}).
\end{array}
\right.
\end{equation}
as in \cref{def:phi},
\[
\eta=c_1(\log X)^{-\frac{2}{3}}(\log\log X)^{-\frac{1}{3}},\quad
\lambda=\frac{\log K}{\log Y},
\]
and constants $A,c_1>0$ and the implicit constant are absolute.
\end{lemma}
\begin{proof}
By \cref{lem:KV} and \cref{lem:HI}, the left-hand side is bounded by
\begin{equation}
\label{zero_sum:sum_parts}
\sum_{\substack{\rho\\K<|\gamma|\le2K\\\beta\ge\frac{1}{2}}}Y^{\beta}
=-\int_{\frac{1}{2}}^{1-\eta}Y^{\alpha}dN(\alpha,2K)
\ll
KY^{\frac{1}{2}}L
+L^A\int_{\frac{1}{2}}^{1-\eta}K^{c(\alpha)}Y^{\alpha}d\alpha
\end{equation}
for sufficiently small $c_1>0$.
We determine the maximum value of 
\[
K^{c(\alpha)}Y^{\alpha}
=
Y^{\lambda c(\alpha)+\alpha}
\]
over $\alpha\in[\frac{1}{2},1-\eta]$. Let $h(\alpha)=\lambda c(\alpha)+\alpha$.
For $\alpha\in[\frac{1}{2},\frac{3}{4}]$, we have
\[
h(\alpha)
=
\frac{3\lambda(1-\alpha)}{2-\alpha}+\alpha
=
3\lambda-\frac{3\lambda}{2-\alpha}+\alpha.
\]
By taking the derivative,
\[
h'(\alpha)=-\frac{3\lambda}{(2-\alpha)^2}+1.
\]
Thus, in the range $\alpha\in(-\infty,2)$,
\[
h'(\alpha)=0\quad\Longleftrightarrow\quad
\alpha=2-\sqrt{3\lambda},
\]
so $h(\alpha)$ is increasing for $\alpha<2-\sqrt{3\lambda}$
and decreasing for $2-\sqrt{3\lambda}<\alpha<2$.
Hence,
\[
\max_{\alpha\in[\frac{1}{2},\frac{3}{4}]}h(\alpha)
=
\left\{
\begin{array}{>{\displaystyle}ll}
\frac{3}{5}\lambda+\frac{3}{4}&(\text{if $0\le\lambda\le\frac{25}{48}$}),\\[3mm]
3\lambda+2(1-\sqrt{3\lambda})&(\text{if $\frac{25}{48}\le\lambda\le\frac{3}{4}$}),\\[3mm]
\lambda+\frac{1}{2}&(\text{if $\frac{3}{4}\le\lambda\le1$}).
\end{array}
\right.
\]
For $\alpha\in[\frac{3}{4},1-\eta]$, we have
\[
h(\alpha)
=
\frac{3\lambda(1-\alpha)}{3\alpha-1}+\alpha
=
-\lambda+\frac{2\lambda}{3\alpha-1}+\alpha.
\]
By taking the derivative twice, in the range $\alpha\in[\frac{3}{4},1-\eta]$,
\[
h''(\alpha)=\frac{18\lambda}{(3\alpha-1)^3}>0
\]
so that $h(\alpha)$ is convex downwards in this range.
Thus, for small $c_1$,
\[
\max_{\alpha\in[\frac{3}{4},1-\eta]}h(\alpha)
=
\max\left(h\left(\frac{3}{4}\right),h\left(1-\eta\right)\right)
\le
\max\left(h\left(\frac{3}{4}\right),1-\eta+2\eta\lambda\right).
\]
By using the above observations for $h(\alpha)$ in \cref{zero_sum:sum_parts},
we obtain the lemma.
\end{proof}

\begin{lemma}
\label{lem:phi_dev}
Let $\phi(\lambda)$ be the function given by \cref{def:phi},
Then,
\[
\frac{3}{5}\le\phi'(\lambda)\le1
\]
for $\lambda\ge0$.
In particular, $\phi(\lambda)$ is increasing.
\end{lemma}
\begin{proof}
It suffices to consider the case $\frac{25}{48}\le\lambda\le\frac{3}{4}$.
In this range,
\[
\phi'(\lambda)=3-\sqrt{\frac{3}{\lambda}}.
\]
Thus, the lemma easily follows.
\end{proof}

\begin{lemma}
\label{lem:phi_solve}
Let $\phi(\lambda)$ be the function given by \cref{def:phi}.
For real numbers $k,\ell$ with $k\ge1$ and $\ell\ge2$,
consider the solutions $\lambda_1$ and $\lambda_2$ of the equations
\begin{equation}
\label{phi_solve:eq}
\phi(\lambda_1)-\frac{1}{\ell}\lambda_1=1,\quad
\phi(\lambda_2)+\frac{1}{2}\lambda_2=1+\frac{k}{\ell}.
\end{equation}
Then, these functions $\lambda_1,\lambda_2$
are consistent with the functions given in \cref{def:lambda}.
\end{lemma}
\begin{proof}
By \cref{lem:phi_dev} and $\ell\ge2$, both of the continuous functions
\begin{equation}
\label{phi_solve:func}
\phi(\lambda)-\frac{1}{\ell}\lambda,\quad
\phi(\lambda)+\frac{1}{2}\lambda
\end{equation}
are strictly increasing for $\lambda\ge0$ and
take the value from $3/4$ to $+\infty$.
Thus, by the intermediate value theorem, $\lambda_1$ and $\lambda_2$ are well-defined.

We first consider $\lambda_1$.
If $\phi(\frac{25}{48})-\frac{25}{48\ell}\ge1$, i.e.~$\ell\ge\frac{25}{3}$, then
\[
1
=
\phi(\lambda_1)-\frac{1}{\ell}\lambda_1
=
\left(\frac{3}{5}-\frac{1}{\ell}\right)\lambda_1+\frac{3}{4}
\]
so that
\[
\lambda_1=\frac{5\ell}{4(3\ell-5)}.
\]
If $\phi(\frac{25}{48})-\frac{1}{\ell}\frac{25}{48}\le1\le\phi(\frac{3}{4})-\frac{3}{4\ell}$,
i.e.~$3\le\ell\le\frac{25}{3}$, then
\[
1
=
\phi(\lambda_1)-\frac{1}{\ell}\lambda_1
=
\left(3-\frac{1}{\ell}\right)\lambda_1+2(1-\sqrt{3\lambda_1})
\]
so that, by using $\frac{25}{48}\le\lambda_1$ in the current case,
\[
\lambda_1
=
\frac{3\ell^2+2\sqrt{3}\ell^{\frac{3}{2}}+\ell}{(3\ell-1)^2}.
\]
Finally, if $\phi(\frac{3}{4})-\frac{3}{4\ell}\le1$, i.e.~$2\le\ell\le3$,
then
\[
1
=
\phi(\lambda_1)-\frac{1}{\ell}\lambda_1
=
\left(1-\frac{1}{\ell}\right)\lambda_1+\frac{1}{2}
\]
so that
\[
\lambda_1=\frac{\ell}{2(\ell-1)}.
\]
This completes the proof of the assertion for $\lambda_1$.

We next consider $\lambda_2$.
If $1+\frac{k}{\ell}\le\phi(\frac{25}{48})+\frac{1}{2}\cdot\frac{25}{48}$,
i.e.~$k\le\frac{31}{96}\ell$, then
\[
1+\frac{k}{\ell}
=
\phi(\lambda_2)+\frac{1}{2}\lambda_2
=
\frac{11}{10}\lambda_2+\frac{3}{4}
\]
so that
\[
\lambda_2=\frac{10}{11}\left(\frac{k}{\ell}+\frac{1}{4}\right).
\]
If $\phi(\frac{25}{48})+\frac{1}{2}\cdot\frac{25}{48}
\le1+\frac{k}{\ell}\le\phi(\frac{3}{4})+\frac{1}{2}\cdot\frac{3}{4}$,
i.e.~$\frac{31}{96}\ell\le k\le\frac{5}{8}\ell$, then
\[
1+\frac{k}{\ell}
=
\phi(\lambda_2)+\frac{1}{2}\lambda_2
=
\frac{7}{2}\lambda_2+2(1-\sqrt{3\lambda_2})
\]
so that, by using $\frac{25}{48}\le\lambda_2$ in the current case,
\[
\lambda_2
=\frac{10}{49}+\frac{2k}{7\ell}+\frac{4}{7}\sqrt{\frac{6}{7}\left(\frac{k}{\ell}-\frac{1}{7}\right)}.
\]
Finally, if $\phi(\frac{3}{4})+\frac{1}{2}\cdot\frac{3}{4}
\le1+\frac{k}{\ell}$, i.e.~$\frac{5}{8}\ell\le k$,
then
\[
1+\frac{k}{\ell}
=
\phi(\lambda_2)+\frac{1}{2}\lambda_2
=
\frac{3}{2}\lambda_2+\frac{1}{2}
\]
so that
\[
\lambda_2=\frac{2}{3}\left(\frac{k}{\ell}+\frac{1}{2}\right).
\]
This completes the proof of the assertion for $\lambda_2$.
\end{proof}

\begin{lemma}
\label{lem:phi_solve_shift}
For positive integers $k,\ell$ with $\ell\ge2$
and a real number $\epsilon$ with $\epsilon>0$,
\[
\phi(\lambda)-\frac{1}{\ell}\lambda\le1-\frac{\epsilon}{10}\quad\text{and}\quad
\phi(\lambda)+\frac{1}{2}\lambda\le1+\frac{k}{\ell}-\frac{\epsilon}{10}
\]
provided
\begin{equation}
\label{phi_solve_shift:cond}
0\le\lambda\le\min(\lambda_1,\lambda_2)-\epsilon,
\end{equation}
where $\lambda_1,\lambda_2$ are the solutions of \cref{phi_solve:eq},
or equivalently, defined by \cref{def:lambda}.
\end{lemma}
\begin{proof}
By the assumption $\ell\ge2$ and \cref{lem:phi_dev},
both of the functions \cref{phi_solve:func}
have the derivative of the size $\ge\frac{1}{10}$.
Thus, the mean value theorem and \cref{phi_solve_shift:cond} give
\[
\phi(\lambda)-\frac{1}{\ell}\lambda
\le
\phi(\lambda_1)-\frac{1}{\ell}\lambda_1
-\frac{1}{10}(\lambda_1-\lambda)
\le
1-\frac{\epsilon}{10}
\]
and
\[
\phi(\lambda)+\frac{1}{2}\lambda
\le
\phi(\lambda_2)+\frac{1}{2}\lambda_2
-\frac{1}{10}(\lambda_2-\lambda)
\le
1+\frac{k}{\ell}-\frac{\epsilon}{10}.
\]
This completes the proof.
\end{proof}

\begin{lemma}
\label{lem:monotonicity_lambda}
The functions $\lambda_1(\ell),\lambda_2(k,\ell)$
are strictly decreasing with respect to $\ell$.
\end{lemma}
\begin{proof}
By \cref{lem:phi_solve}, we $\lambda_1(\ell)$ and $\lambda_2(k,\ell)$
can be regarded as the solutions of the equations \cref{phi_solve:eq}.
Then, the lemma follows since the functions \cref{phi_solve:func} are increasing.
\end{proof}

As we mentioned in \cref{section:intro},
we shall apply the Poisson summation formula in order to detect
some cancellation over the sequence $n^{\ell}$.
In order to estimate the resulting exponential integrals,
we recall the next two standard estimates.

\begin{lemma}[First derivative estimate]
\label{lem:first_dev}
Let $\lambda$ be a positive real number,
and $f,g$ be real-valued functions defined over an interval $[a,b]$ satisfying
\begin{enumerate}
\renewcommand{\labelenumi}{\upshape(\Alph{enumi})}
\item $f$ is continuously differentiable on the interval $[a,b]$,
\item $f'$ is monotonic on the interval $[a,b]$, and
\item $f'$ satisfies $|f'(x)|\ge\lambda$ on the interval $[a,b]$.
\end{enumerate}
Then, by using notation \cref{def:norm}, we have
\[
\int_a^bg(x)e(f(x))dx\ll\|g\|\lambda^{-1},
\]
where the implicit constant is absolute.
\end{lemma}
\begin{proof}
See \cite[Lemma 2.1, p.~56]{Ivic}.
\end{proof}

\begin{lemma}[Second derivative estimate]
\label{lem:second_dev}
Let $\lambda$ be a positive real number,
and $f,g$ be real-valued functions defined over an interval $[a,b]$ satisfying
\begin{enumerate}
\renewcommand{\labelenumi}{\upshape(\Alph{enumi})}
\item $f$ is twice continuously differentiable on the interval $[a,b]$,
\item $f''$ satisfies $|f''(x)|\ge\lambda$ on the interval $[a,b]$.
\end{enumerate}
Then, by using notation \cref{def:norm}, we have
\[
\int_a^bg(x)e(f(x))dx\ll\|g\|\lambda^{-\frac{1}{2}},
\]
where the implicit constant is absolute.
\end{lemma}
\begin{proof}
See \cite[Lemma 2.2, p.~56]{Ivic}.
\end{proof}

\section{Preliminary calculations}
In this section, we carry out preliminary calculations for the proof of \cref{thm:main_kl}.
We first replace $\log p$ in \cref{def:R} by the von Mangoldt function.

\begin{lemma}
\label{lem:to_Lambda}
For positive integers $k,\ell$
and real numbers $X,H,\epsilon$ with $4\le H\le X$ and $\epsilon>0$,
we have
\[
\sum_{X<N\le X+H}R(N)
=
\sum_{X<m^{k}+n^{\ell}\le X+H}\Lambda(m)
+
O(HX^{\frac{1}{k}+\frac{1}{\ell}-1}B^{-1})
\]
provided
\begin{equation}
\label{to_Lambda:H_cond}
X^{1-\min(\frac{1}{k},\frac{k}{\ell(k-1)})+\epsilon}\le H\le X^{1-\epsilon},
\end{equation}
where the implicit constant is absolute.
\end{lemma}
\begin{proof}
By definition \cref{def:R} of $R(N)$,
\begin{equation}
\label{to_Lambda:first}
\begin{aligned}
\sum_{X<N\le X+H}R(N)
&=
\sum_{X<p^{k}+n^{\ell}\le X+H}\log p\\
&=
\sum_{X<m^{k}+n^{\ell}\le X+H}\Lambda(m)
-
\sum_{\nu=2}^{O(L)}\sum_{X<p^{\nu k}+n^{\ell}\le X+H}\log p.
\end{aligned}
\end{equation}
Note that the implicit constant in \cref{lem:beta_short} is absolute.
Therefore, by \cref{lem:beta_short},
the second term on the right hand side is bounded by
\[
\ll
L\sum_{\nu=2}^{O(L)}(HX^{\frac{1}{\nu k}+\frac{1}{\ell}-1}+H^{\frac{1}{\nu k}}+X^{\frac{1}{\ell}})
\ll
(HX^{\frac{1}{2k}+\frac{1}{\ell}-1}+H^{\frac{1}{k}}+X^{\frac{1}{\ell}})L^2,
\]
which is $\ll HX^{\frac{1}{k}+\frac{1}{\ell}-1}B^{-1}$
provided \cref{to_Lambda:H_cond}.
This completes the proof.
\end{proof}

We then modify the sum on the right-hand side of \cref{lem:to_Lambda}
in order to insert the explicit formula given by \cref{lem:vonMangoldt}.

\begin{lemma}
\label{lem:red}
For positive integers $k,\ell$
and real numbers $X,H,\epsilon$ with $4\le H\le X$ and $\epsilon>0$,
we have
\begin{align}
&\sum_{X<N\le X+H}R(N)\\
&=
\sum_{n^\ell\le X}
\left(\psi\left((X+H-n^{\ell})^{\frac{1}{k}}\right)-\psi\left((X-n^{\ell})^{\frac{1}{k}}\right)\right)
+
O(HX^{\frac{1}{k}+\frac{1}{\ell}-1}B^{-1})
\end{align}
provided
\begin{equation}
\label{red:H_cond}
X^{1-\min(\frac{1}{k},\frac{k}{\ell(k-1)})+\epsilon}\le H\le X^{1-\epsilon},
\end{equation}
where the implicit constant is absolute.
\end{lemma}
\begin{proof}
We truncate the summation over $n$ in \cref{lem:to_Lambda}.
By using \cref{lem:power_short} and the argument
similar to the beginning of the proof of \cref{lem:beta_short},
\[
\sum_{X<n^\ell\le X+H}\sum_{X-n^\ell<m^k\le X+H-n^\ell}\Lambda(m)
\ll
H^{1+\frac{1}{k}}X^{\frac{1}{\ell}-1}+H^{\frac{1}{k}}
\ll
HX^{\frac{1}{k}+\frac{1}{\ell}-1}B^{-1}
\]
provided \cref{red:H_cond}.
Thus we can employ the truncation as
\[
\sum_{X<m^{k}+n^{\ell}\le X+H}\Lambda(m)
=
\sum_{n^\ell\le X}\sum_{X-n^\ell<m^k\le X+H-n^\ell}\Lambda(m)
+
O(HX^{\frac{1}{k}+\frac{1}{\ell}-1}B^{-1}).
\]
By recalling the notation \cref{def:psi}, we arrive at
\begin{equation}
\begin{aligned}
&\sum_{X<m^{k}+n^{\ell}\le X+H}\Lambda(m)\\
&=
\sum_{n^\ell\le X}
\left(\psi\left((X+H-n^{\ell})^{\frac{1}{k}}\right)-\psi\left((X-n^{\ell})^{\frac{1}{k}}\right)\right)
+
O(HX^{\frac{1}{k}+\frac{1}{\ell}-1}B^{-1}).
\end{aligned}
\end{equation}
By substituting this formula into \cref{lem:to_Lambda},
we obtain the lemma.
\end{proof}

\section{Detection of the cancellation over the $\ell$-th powers}
In this section, we derive an expansion for the sum
\[
\sum_{n^\ell\le X}\psi\left((Q-n^\ell)^{\frac{1}{k}}\right),\quad
X\le Q\le X+H,
\]
or its difference
\begin{equation}
\label{S_diff}
\sum_{n^\ell\le X}\left(\psi\left((X+H-n^\ell)^{\frac{1}{k}}\right)-\psi\left((X-n^\ell)^{\frac{1}{k}}\right)\right)
\end{equation}
by which we try to detect some cancellation caused by the average over $n^{\ell}$.
This expansion will be given by \cref{lem:S_explicit_pre} and \cref{lem:S_Poisson}.
We first substitute \cref{lem:vonMangoldt}.
\begin{lemma}
\label{lem:S_explicit_pre}
Let $k,\ell$ be positive integers, and $X,H,Q,T$ be real numbers
satisfying $4\le H\le X$, $X\le Q\le X+H$ and $1\le T\le X^{\frac{1}{k}}$. Then,
\begin{equation}
\label{explicit_pre}
\sum_{n^\ell\le X}\psi\left((Q-n^\ell)^{\frac{1}{k}}\right)
=
S(Q)-\sum_{\substack{\rho\\|\gamma|\le T}}S_{\rho}(Q)+O(X^{\frac{1}{k}+\frac{1}{\ell}}T^{-1}L^2),
\end{equation}
where $S(Q)$ and $S_{\rho}(Q)$ are given by
\[
S(Q)=\sum_{n^\ell\le X}(Q-n^\ell)^{\frac{1}{k}},\quad
S_{\rho}(Q)=\frac{1}{\rho}\sum_{n^\ell\le X}(Q-n^\ell)^{\frac{\rho}{k}}
\]
as defined in \cref{def:S}, and the implicit constant is absolute.
\end{lemma}
\begin{proof}
This follows immediately by inserting \cref{lem:vonMangoldt}.
\end{proof}

Our next task is to detect the cancellation in the sum $S_{\rho}(Q)$.
We prepare the next lemma in order to estimate exponential integrals.

\begin{lemma}
\label{lem:exp_int}
For positive integers $k,\ell$, an integer $n$ not necessarily positive,
and real numbers $\alpha,\gamma,Q,U,V$ with $\alpha\le1$, $|\gamma|\ge1$
and $1\le U\le V\le Q$, we have
\[
\int_{U}^{V}u^{\alpha+\frac{i\gamma}{k}-1}e\left(n(Q-u)^{\frac{1}{\ell}}\right)du
\ll
\left\{
\begin{array}{ll}
\displaystyle
\frac{V^{\alpha}L}{|\gamma|^{\frac{1}{2}}}&(\text{if $\alpha\ge0$}),\\[4mm]
\displaystyle
\frac{U^{\alpha}L}{|\gamma|^{\frac{1}{2}}}&(\text{if $\alpha\le0$}),\\[4mm]
\displaystyle
\frac{Q^{1-\frac{1}{\ell}}}{|n|}&(\text{if $|n|>\ell Q^{1-\frac{1}{\ell}}|\gamma|$}),\\
\end{array}
\right.
\]
where the implicit constant depends on $k,\ell$ and $\epsilon$.
\end{lemma}
\begin{proof}
We rewrite the left-hand side as
\begin{equation}
\label{exp_int:first}
\int_{U}^{V}u^{\alpha+\frac{i\gamma}{k}-1}e\left(n(Q-u)^{\frac{1}{\ell}}\right)du
=
\int_{U}^{V}G(u)e(F(u))du,
\end{equation}
where
\[
F(u)=n(Q-u)^{\frac{1}{\ell}}+\frac{\gamma}{2\pi k}\log u,\quad
G(u)=u^{\alpha-1}.
\]
Then,
\begin{equation}
\label{exp_int:dev}
F'(u)=-\frac{1}{\ell}n(Q-u)^{\frac{1}{\ell}-1}+\frac{\gamma}{2\pi ku},\quad
F''(u)=-\frac{\ell-1}{\ell^2}n(Q-u)^{\frac{1}{\ell}-2}-\frac{\gamma}{2\pi ku^2}
\end{equation}
and since $G(u)$ is non-increasing, by using the notation \cref{def:norm},
\[
\|G\|_{BV([R,R'])}\ll R^{\alpha-1}
\]
for any subinterval $[R,R']\subset[U,V]$.

For the former two estimates, we dissect the integral \cref{exp_int:first} dyadically as
\begin{equation}
\label{exp_int:dyadic}
\ll
L\sup_{U<R\le V}\left|%
\int_{R}^{\min(2R,V)}G(u)e(F(u))du\right|.
\end{equation}
If $n$ and $\gamma$ have the same signs, then we have
\[
|F''(u)|\ge\frac{|\gamma|}{2\pi k(2R)^2}
\]
for $u\in[R,\min(2R,V)]$. Therefore, by \cref{lem:second_dev},
\begin{equation}
\label{exp_int:first_dev}
\int_{R}^{\min(2R,V)}u^{\alpha+\frac{i\gamma}{k}-1}e\left(n(Q-u)^{\frac{1}{\ell}}\right)du
\ll
R^{\alpha-1}\left(\frac{|\gamma|}{R^2}\right)^{-\frac{1}{2}}
\ll
\frac{R^{\alpha}}{|\gamma|^{\frac{1}{2}}}.
\end{equation}
On the other hand, if $n$ and $\gamma$ have the opposite signs, then we have
\[
|F'(u)|\ge\frac{|\gamma|}{2\pi k(2R)}
\]
and $F''(u)$ has at most one zero in $[R,\min(2R,V)]$.
Therefore, we may dissect $[R,\min(2R,V)]$ into at most two intervals,
on each of which $F'(u)$ is monotonic.
By applying \cref{lem:first_dev},
\begin{equation}
\label{exp_int:second_dev}
\int_{R}^{\min(2R,V)}u^{\alpha+\frac{i\gamma}{k}-1}e\left(n(Q-u)^{\frac{1}{\ell}}\right)du
\ll
R^{\alpha-1}\left(\frac{|\gamma|}{R}\right)^{-1}
=
\frac{R^{\alpha}}{|\gamma|}
\ll
\frac{R^{\alpha}}{|\gamma|^{\frac{1}{2}}}
\end{equation}
since $|\gamma|\ge1$.
Therefore, by \cref{exp_int:first_dev} and \cref{exp_int:second_dev}, we have
\[
\int_{R}^{\min(2R,V)}u^{\alpha+\frac{i\gamma}{k}-1}e\left(n(Q-u)^{\frac{1}{\ell}}\right)du
\ll
\frac{R^{\alpha}}{|\gamma|^{\frac{1}{2}}}
\]
in any case. On inserting this estimate into \cref{exp_int:dyadic},
we obtain the first two estimates.

For the last estimate, we work without the dyadic dissection.
We apply \cref{lem:first_dev} to the integral \cref{exp_int:first}.
By assuming $|n|>\ell Q^{1-\frac{1}{\ell}}|\gamma|$,
\[
|F'(u)|
\ge
\frac{|n|}{\ell Q^{1-\frac{1}{\ell}}}-\frac{|\gamma|}{2\pi k}
\gg
\frac{|n|}{Q^{1-\frac{1}{\ell}}}.
\]
Also, by \cref{exp_int:dev}, we can dissect $[U,V]$
into at most two intervals, on each of which $F'(u)$ is monotonic.
Thus, by \cref{lem:first_dev},
\[
\int_{U}^{V}u^{\alpha+\frac{i\gamma}{k}-1}e\left(n(Q-u)^{\frac{1}{\ell}}\right)du
\ll
U^{\alpha-1}
\left(\frac{|n|}{Q^{1-\frac{1}{\ell}}}\right)^{-1}
\ll
\frac{Q^{1-\frac{1}{\ell}}}{|n|}
\]
since $\alpha\le 1$. This completes the proof.
\end{proof}

We now apply the Poisson summation formula
and detect the cancellation over the sequence $n^{\ell}$.

\begin{lemma}
\label{lem:S_Poisson}
For positive integers $k,\ell$,
real numbers $X,H,\epsilon$ with $4\le H\le X$ and $\epsilon>0$,
and a non-trivial zero $\rho=\beta+i\gamma$ of $\zeta(s)$
with $|\gamma|\le2X$, we have
\begin{equation}
\label{S_expansion}
\begin{aligned}
&S_{\rho}(X+H)-S_{\rho}(X)\\
&=
\GGG{\rho}
\left((X+H)^{\frac{\rho}{k}+\frac{1}{\ell}}-X^{\frac{\rho}{k}+\frac{1}{\ell}}\right)\\
&\quad-\frac{(X+H)^{\frac{\rho}{k}}-X^{\frac{\rho}{k}}}{2\rho}
+O\left(H^{\frac{\beta}{k}}|\gamma|^{\frac{\beta}{k}-\frac{1}{2}}L^2
+\frac{HX^{\frac{1}{k}+\frac{1}{\ell}-1}B^{-2}}{|\gamma|}+L\right)
\end{aligned}
\end{equation}
provided
\begin{equation}
\label{S_Poisson:H_cond}
X^{1-\min(\frac{1}{k},\frac{k}{\ell(k-1)})+\epsilon}\le H\le X^{1-\epsilon},
\end{equation}
where the implicit constant depends on $k,\ell$ and $\epsilon$.
\end{lemma}
\begin{proof}
By partial summation, for $X\le Q\le X+H$, we have
\begin{equation}
\label{S_Poisson:sum_parts}
\begin{aligned}
S_{\rho}(Q)
&=
\frac{1}{\rho}\int_0^{X}(Q-u)^{\frac{\rho}{k}}d[u^{\frac{1}{\ell}}]\\
&=
\frac{1}{\ell\rho}\int_0^{X}(Q-u)^{\frac{\rho}{k}}u^{\frac{1}{\ell}-1}du
-
\frac{1}{\rho}\int_0^{X}(Q-u)^{\frac{\rho}{k}}d\left(\{u^{\frac{1}{\ell}}\}-\frac{1}{2}\right).
\end{aligned}
\end{equation}
The first integral on the right-hand side of \cref{S_Poisson:sum_parts} is
\begin{align}
\frac{1}{\ell\rho}\int_0^{X}(Q-u)^{\frac{\rho}{k}}u^{\frac{1}{\ell}-1}du
&=
\frac{1}{\ell\rho}\int_0^{Q}(Q-u)^{\frac{\rho}{k}}u^{\frac{1}{\ell}-1}du
+O\left(\frac{H^{1+\frac{1}{k}}X^{\frac{1}{\ell}-1}}{|\gamma|}\right)\\
&=
\GGG{\rho}Q^{\frac{\rho}{k}+\frac{1}{\ell}}
+O\left(\frac{HX^{\frac{1}{k}+\frac{1}{\ell}-1}B^{-2}}{|\gamma|}\right)
\end{align}
provided \cref{S_Poisson:H_cond}.
The second integral on the right-hand side of \cref{S_Poisson:sum_parts} is
\begin{align}
&-\frac{1}{\rho}\int_0^{X}(Q-u)^{\frac{\rho}{k}}d\left(\{u^{\frac{1}{\ell}}\}-\frac{1}{2}\right)\\
&=
-\frac{1}{k}\int_0^{X}(Q-u)^{\frac{\rho}{k}-1}\left(\{u^{\frac{1}{\ell}}\}-\frac{1}{2}\right)du
-\frac{Q^{\frac{\rho}{k}}}{2\rho}
+O\left(\frac{H^{\frac{1}{k}}}{|\gamma|}\right)\\
&=
-\frac{1}{k}\int_0^{X}(Q-u)^{\frac{\rho}{k}-1}\left(\{u^{\frac{1}{\ell}}\}-\frac{1}{2}\right)du
-\frac{Q^{\frac{\rho}{k}}}{2\rho}
+O\left(\frac{HX^{\frac{1}{k}+\frac{1}{\ell}-1}B^{-2}}{|\gamma|}\right)
\end{align}
provided \cref{S_Poisson:H_cond}.
Recall the Fourier expansion
\[
\{u\}-\frac{1}{2}
=
-\sum_{n\neq0}\frac{e(nu)}{2\pi in},
\]
which holds for $u\not\in\mathbb{Z}$ and converges boundedly for $u\in\mathbb{R}$.
Then since
\begin{align}
&-\frac{1}{k}\int_0^{X}(Q-u)^{\frac{\rho}{k}-1}\left(\{u^{\frac{1}{\ell}}\}-\frac{1}{2}\right)du\\
&=
-\frac{1}{k}\int_0^{X-1}(Q-u)^{\frac{\rho}{k}-1}\left(\{u^{\frac{1}{\ell}}\}-\frac{1}{2}\right)du
+O\left(\frac{1}{k}\int_{X-1}^{X}(Q-u)^{\frac{\beta}{k}-1}du\right)\\
&=
-\frac{1}{k}\int_0^{X-1}(Q-u)^{\frac{\rho}{k}-1}\left(\{u^{\frac{1}{\ell}}\}-\frac{1}{2}\right)du
+O\left(\frac{1}{\beta}\right),
\end{align}
by using \cref{lem:KV} and the assumption $|\gamma|\le X$, we have
\begin{equation}
\label{S_Poisson:pre_S_expansion}
S_{\rho}(Q)
=
\GGG{\rho}
Q^{\frac{\rho}{k}+\frac{1}{\ell}}
-\frac{Q^{\frac{\rho}{k}}}{2\rho}
+R_\rho(Q)+O\left(\frac{HX^{\frac{1}{k}+\frac{1}{\ell}-1}B^{-2}}{|\gamma|}+L\right)
\end{equation}
for $X\le Q\le X+H$, where
\begin{gather}
R_\rho(Q)=R_{\rho,k,\ell}(Q)=\sum_{n\neq0}\frac{I_\rho(Q,n)}{2\pi ikn},\\
I_\rho(Q,n)=I_{\rho,k,\ell}(Q,n)
=\int_{0}^{X-1}(Q-u)^{\frac{\rho}{k}-1}e(nu^\frac{1}{\ell})du.
\end{gather}
In order to prove the lemma, it suffices to estimate
\[
R_{\rho}(X+H)-R_{\rho}(X).
\]
We first estimate the difference of oscillating integrals
\begin{equation}
\label{S_expansion:I_diff}
I_{\rho}(X+H,n)-I_{\rho}(X,n).
\end{equation}
By changing the variable in the definiton of $I_{\rho}(Q,n)$,
we obtain expressions
\begin{align}
I_{\rho}(X+H,n)
&=\int_1^{X}(u+H)^{\frac{\rho}{k}-1}e\left(n(X-u)^{\frac{1}{\ell}}\right)du,\\
I_{\rho}(X,n)
&=\int_1^{X}u^{\frac{\rho}{k}-1}e\left(n(X-u)^{\frac{1}{\ell}}\right)du.
\end{align}
Let $U=\min(4H|\gamma|,X)$.
Then we decompose \cref{S_expansion:I_diff} as
\begin{equation}
\label{S_Poisson:I_decomp}
I_{\rho}(X+H,n)-I_{\rho}(X,n)
=
I+I_1-I_2,
\end{equation}
where
\[
I
=
\int_{U}^{X}\left((u+H)^{\frac{\rho}{k}-1}-u^{\frac{\rho}{k}-1}\right)
e\left(n(X-u)^{\frac{1}{\ell}}\right)du,
\]
\[
I_1=\int_{1}^{U}(u+H)^{\frac{\rho}{k}-1}e\left(n(X-u)^{\frac{1}{\ell}}\right)du,\quad
I_2=\int_{1}^{U}u^{\frac{\rho}{k}-1}e\left(n(X-u)^{\frac{1}{\ell}}\right)du.
\]
For the integral $I$, we use the Taylor expansion
\[
(u+H)^{\frac{\rho}{k}-1}-u^{\frac{\rho}{k}-1}
=
u^{\frac{\rho}{k}-1}
\sum_{\nu=1}^{\infty}\binom{\frac{\rho}{k}-1}{\nu}\left(\frac{H}{u}\right)^{\nu}.
\]
By substituting this expansion into the definition of $I$,
\begin{equation}
\label{S_Poisson:I_expansion}
I
=
\sum_{\nu=1}^{\infty}
\binom{\frac{\rho}{k}-1}{\nu}
H^{\nu}\int_{U}^{X}u^{\frac{\rho}{k}-\nu-1}
e\left(n(X-u)^{\frac{1}{\ell}}\right)du.
\end{equation}
By using \cref{lem:exp_int} and the definition of $U$, if $4H|\gamma|\le X$,
\begin{equation}
\label{S_Poisson:I}
I
\ll
\frac{U^{\frac{\beta}{k}}L}{|\gamma|^{\frac{1}{2}}}
\sum_{\nu=1}^{\infty}
\left|\binom{\frac{\rho}{k}-1}{\nu}\right|
\left(\frac{H}{U}\right)^{\nu}
\ll
\frac{U^{\frac{\beta}{k}}L}{|\gamma|^{\frac{1}{2}}}
\sum_{\nu=1}^{\infty}
\prod_{\mu=1}^{\nu}\left(\frac{|\gamma|+2\mu}{4\mu|\gamma|}\right)
\ll
\frac{U^{\frac{\beta}{k}}L}{|\gamma|^{\frac{1}{2}}}
\end{equation}
since $|\gamma|\ge2$.
If $4H|\gamma|>X$, then $I$ is an empty integral,
so the same estimate holds trivially.
For the integrals $I_1$ and $I_2$, we may use \cref{lem:exp_int} directly to obtain
\begin{equation}
\label{S_Poisson:I12}
I_1,I_2
\ll
\frac{U^{\frac{\beta}{k}}L}{|\gamma|^{\frac{1}{2}}}
\end{equation}
since we can choose $Q=X+H$ for the integral
\begin{align}
I_1
&=
\int_{1}^{U}(u+H)^{\frac{\beta}{k}+\frac{i\gamma}{k}-1}e\left(n(X-u)^{\frac{1}{\ell}}\right)du\\
&=
\int_{1+H}^{U+H}u^{\frac{\beta}{k}+\frac{i\gamma}{k}-1}e\left(n(X+H-u)^{\frac{1}{\ell}}\right)du.
\end{align}
Therefore, we have
\begin{equation}
\label{S_Poisson:n_small}
I_{\rho}(X+H,n)-I_{\rho}(X,n)
\ll
\frac{U^{\frac{\beta}{k}}L}{|\gamma|^{\frac{1}{2}}}
\ll
H^{\frac{\beta}{k}}|\gamma|^{\frac{\beta}{k}-\frac{1}{2}}L.
\end{equation}
On the other hand, if $|n|>\ell (X+H)^{1-\frac{1}{\ell}}|\gamma|$,
\cref{lem:exp_int} gives
\[
I_{\rho}(X+H,n)-I_{\rho}(X,n)
\ll
\frac{X^{1-\frac{1}{\ell}}}{|n|}.
\]
Thus we have
\begin{equation}
\begin{aligned}
&R_{\rho}(X+H)-R_{\rho}(X)\\
&\ll
H^{\frac{\beta}{k}}|\gamma|^{\frac{\beta}{k}-\frac{1}{2}}L
\sum_{n\le\ell (X+H)^{1-\frac{1}{\ell}}|\gamma|}\frac{1}{n}
+
X^{1-\frac{1}{\ell}}\sum_{n>\ell (X+H)^{1-\frac{1}{\ell}}|\gamma|}\frac{1}{n^2}\\
&\ll
H^{\frac{\beta}{k}}|\gamma|^{\frac{\beta}{k}-\frac{1}{2}}L^2
+
1.
\end{aligned}
\end{equation}
This completes the proof.
\end{proof}

\section{Completion of the proof}
\label{section:proof}
In this section, we complete the proof of main theorems.
However, before the main part of the proof of \cref{thm:main_kl},
we check the direct consequence of \cref{lem:Huxley_PNT}.

\begin{lemma}
\label{lem:trivial}
For positive integers $k,\ell$ with $\ell\ge2$
and real numbers $X,H,\epsilon$
with $4\le H\le X$ and $\epsilon>0$,
we have the asymptotic formula \cref{LZ_asymp_kl}
provided
\begin{equation}
\label{trivial:H_cond}
X^{1-\theta_{B}(k,\ell)+\epsilon}\le H\le X^{1-\epsilon},
\end{equation}
where $\theta_{B}(k,\ell)$ is defined by
\[
\theta_{B}(k,\ell)
=
\min\left(\frac{5}{12k},\frac{k}{\ell(k-1)}\right)
\]
as in \cref{thm:main_kl} and the implicit constant depends on $k,\ell$ and $\epsilon$.
\end{lemma}
\begin{proof}
We may assume that $X$ is larger than some constant
depends only on $k,\ell$ and $\epsilon$ since otherwise
the assertion trivially holds.
We use \cref{lem:Huxley_PNT} in \cref{lem:red}.
If $n^{\ell}\le X$ and $(X+H-n^{\ell})\le2(X-n^{\ell})$,
\begin{align}
(X+H-n^{\ell})^{\frac{1}{k}}-(X-n^{\ell})^{\frac{1}{k}}
&=
\frac{1}{k}\int_{X-n^{\ell}}^{X+H-n^{\ell}}u^{\frac{1}{k}-1}du
\ge
\frac{1}{2k}H(X-n^{\ell})^{\frac{1}{k}-1}\\
&\ge
\frac{1}{2k}X^{1-\frac{5}{12k}+\epsilon}(X-n^{\ell})^{\frac{1}{k}-1}
\ge
\left((X-n^{\ell})^{\frac{1}{k}}\right)^{\frac{7}{12}+\frac{\epsilon}{2}}
\end{align}
provided \cref{trivial:H_cond} and $X$ is large. Thus, in this case, \cref{lem:Huxley_PNT} gives
\begin{equation}
\label{trivial:after_Huxley_PNT}
\begin{aligned}
&\psi\left((X+H-n^{\ell})^{\frac{1}{k}}\right)
-
\psi\left((X-n^{\ell})^{\frac{1}{k}}\right)\\
&=
(X+H-n^{\ell})^{\frac{1}{k}}-(X-n^{\ell})^{\frac{1}{k}}
+
O(((X+H-n^{\ell})^{\frac{1}{k}}-(X-n^{\ell})^{\frac{1}{k}})B^{-1})
\end{aligned}
\end{equation}
by making the constant $c$ smaller since
\[
(X-n^{\ell})^{\frac{1}{k}}\gg(X+H-n^{\ell})^{\frac{1}{k}}\gg H^{\frac{1}{k}}
\]
in the current case.
If $n^{\ell}\le X$ and $(X+H-n^{\ell})>2(X-n^{\ell})$,
then we may apply the usual prime number theorem
to obtain the same estimate \cref{trivial:after_Huxley_PNT}
since in this case
\[
(X+H-n^{\ell})^{\frac{1}{k}}-(X-n^{\ell})^{\frac{1}{k}}
\asymp
(X+H-n^{\ell})^{\frac{1}{k}}.
\]
By using \cref{trivial:after_Huxley_PNT}
in \cref{lem:red} and using \cref{lem:main_term},
we arrive at the lemma.
\end{proof}

We now prove the main part of \cref{thm:main_kl}.

\begin{lemma}
\label{lem:lemma_kl}
For positive integers $k,\ell$ with $\ell\ge2$
and real numbers $X,H,\epsilon$ with $4\le H\le X$ and $\epsilon>0$,
we have the asymptotic formula \cref{LZ_asymp_kl}
provided
\begin{equation}
\label{lemma_kl:H_cond}
X^{1-\theta_{C}(k,\ell)+\epsilon}\le H\le X^{1-\epsilon},
\end{equation}
where $\theta_{C}(k,\ell)$ is defined by
\[
\theta_{C}(k,\ell)
=
\left\{
\begin{array}{ll}
\displaystyle
\min\left(\frac{\lambda_1(\ell)}{k},\frac{\lambda_2(k,\ell)}{k},\frac{2}{\ell}\right)
&(\text{if $k=1$}),\\[4mm]
\displaystyle
\min\left(\frac{\lambda_1(\ell)}{k},\frac{\lambda_2(k,\ell)}{k},\frac{k}{\ell(k-1)}\right)
&(\text{if $k\ge2$}),\\
\end{array}
\right.
\]
and the implicit constant depends on $k,\ell$ and $\epsilon$.
\end{lemma}
\begin{proof}
We may assume that $X$ is larger than some constant
depends only on $k,\ell$ and $\epsilon$ since otherwise
the assertion trivially holds.
By \cref{lem:red}, \cref{lem:S_explicit_pre} and \cref{lem:S_Poisson},
\begin{equation}
\label{main:decomp}
\begin{aligned}
&\sum_{X<N\le X+H}R(N)\\
&=
M+R_1+R_2
+O((R_3+X^{\frac{1}{k}+\frac{1}{\ell}}T^{-1}+T)L^2+HX^{\frac{1}{k}+\frac{1}{\ell}-1}B^{-1})
\end{aligned}
\end{equation}
provided
\begin{equation}
\label{HT_cond}
X^{1-\min(\frac{1}{k},\frac{k}{\ell(k-1)})+\epsilon}\le H\le X^{1-\epsilon},\quad
2\le T\le X^{\frac{1}{k}},
\end{equation}
where
\begin{align}
M
&=
S(X+H)-S(X),\\
R_1
&=
-
\sum_{\substack{\rho\\|\gamma|\le T}}\GGG{\rho}
\left((X+H)^{\frac{\rho}{k}+\frac{1}{\ell}}-X^{\frac{\rho}{k}+\frac{1}{\ell}}\right),\\
R_2
&=
\sum_{\substack{\rho\\|\gamma|\le T}}\frac{(X+H)^{\frac{\rho}{k}}-X^{\frac{\rho}{k}}}{2\rho},\quad
R_3
=\sum_{|\gamma|\le T}H^{\frac{\beta}{k}}|\gamma|^{\frac{\beta}{k}-\frac{1}{2}}.
\end{align}
In order to control the size of the error $X^{\frac{1}{k}+\frac{1}{\ell}}T^{-1}L^2$,
we choose $T$ by
\begin{equation}
\label{main:T}
T=X^{1+\frac{\epsilon_1}{k}}H^{-1},\quad
0<\epsilon_1\le\frac{\epsilon}{2},
\end{equation}
where we choose $\epsilon_1$ later (our choice will be $\epsilon_1=\frac{\epsilon}{80}$).
This choice is admissible since the former inequality of \cref{HT_cond} implies
\begin{equation}
\label{main:T_range}
X^{\epsilon}\le T\le X^{\frac{1}{k}-\frac{\epsilon}{2}}.
\end{equation}
If we assume further
\begin{equation}
\label{HT_cond2}
X^{1-\frac{1}{2}(\frac{1}{k}+\frac{1}{\ell})+\epsilon}\le H,
\end{equation}
then
\[
TL^2
= X^{1+\frac{\epsilon_1}{k}}H^{-1}L^2
= HX^{1+\frac{\epsilon_1}{k}}H^{-2}L^2
\le HX^{\frac{1}{k}+\frac{1}{\ell}-1-\epsilon}L^2
\ll HX^{\frac{1}{k}+\frac{1}{\ell}-1}B^{-1}.
\]
Thus,
\begin{equation}
\label{main:error}
(X^{\frac{1}{k}+\frac{1}{\ell}}T^{-1}+T)L^2
\ll
HX^{\frac{1}{k}+\frac{1}{\ell}-1}B^{-1}
\end{equation}
provided \cref{HT_cond2}. By \cref{lem:main_term}, the main term $M$ can be evaluated as
\begin{equation}
\label{main:main}
\label{M}
M
=
\GG{1} HX^{\frac{1}{k}+\frac{1}{\ell}-1}+O\left(HX^{\frac{1}{k}+\frac{1}{\ell}-1}B^{-1}\right)
\end{equation}
provided \cref{HT_cond}. The remaining task is to estimate $R_1,R_2$ and $R_3$.

We first estimate the sum $R_1$. By the fundamental theorem of calculus,
\begin{equation}
\label{main:R1_int}
(X+H)^{\frac{\rho}{k}+\frac{1}{\ell}}-X^{\frac{\rho}{k}+\frac{1}{\ell}}
=
\left(\frac{\rho}{k}+\frac{1}{\ell}\right)\int_{X}^{X+H}u^{\frac{\rho}{k}+\frac{1}{\ell}-1}du
\ll
|\gamma|HX^{\frac{\beta}{k}+\frac{1}{\ell}-1}.
\end{equation}
Then, by using Stirling's formula and dissecting dyadically,
\begin{equation}
\label{main:R1_dyadic}
R_{1}
\ll
HX^{\frac{1}{\ell}-1}\sum_{|\gamma|\le T}\frac{X^{\frac{\beta}{k}}}{|\gamma|^{\frac{1}{\ell}}}
\ll
HX^{\frac{1}{\ell}-1}L\sup_{1\le K\le T}
K^{-\frac{1}{\ell}}\sum_{K<|\gamma|\le 2K}X^{\frac{\beta}{k}}.
\end{equation}
For $1\le K\le T$, we write $K=X^{\frac{\delta}{k}}$.
Further, we write
\begin{equation}
\label{main:Delta}
XH^{-1}=X^{\frac{\Delta}{k}}.
\end{equation}
Then, by \cref{main:T}, $\delta$ moves in the range
\begin{equation}
\label{main:delta_range}
0\le\delta\le\Delta+\epsilon_1.
\end{equation}
By \cref{lem:zero_sum},
\begin{equation}
\label{main:R1_zero_sum}
K^{-\frac{1}{\ell}}\sum_{K<|\gamma|\le 2K}X^{\frac{\beta}{k}}
\ll
\left(X^{\frac{1}{k}(\phi(\delta)-\frac{1}{\ell}\delta)}
+X^{\frac{1}{k}(1-\eta+(2\eta-\frac{1}{\ell})\delta)}\right)L^A.
\end{equation}
By \cref{lem:phi_dev} and the assumption $\ell\ge2$, for sufficiently large $X$,
\[
\frac{d}{d\delta}\left(\phi(\delta)-\frac{1}{\ell}\delta\right)>0,\quad
2\eta-\frac{1}{\ell}<0.
\]
Therefore, by \cref{main:R1_dyadic}, \cref{main:delta_range} and \cref{main:R1_zero_sum},
\begin{equation}
\label{main:R1_pre}
\begin{aligned}
R_{1}
&\ll
HX^{\frac{1}{\ell}-1}
\left(X^{\frac{1}{k}(\phi(\Delta+\epsilon_1)-\frac{1}{\ell}(\Delta+\epsilon_1))}
+X^{\frac{1}{k}(1-\eta)}\right)L^{A+1}\\
&\ll
HX^{\frac{1}{\ell}-1+\frac{1}{k}(\phi(\Delta+\epsilon_1)-\frac{1}{\ell}(\Delta+\epsilon_1))}L^{A+1}
+
HX^{\frac{1}{k}+\frac{1}{\ell}-1}B^{-1}.
\end{aligned}
\end{equation}
By \cref{lem:phi_dev} and the mean value theorem,
\[
\phi(\Delta+\epsilon_1)-\frac{1}{\ell}(\Delta+\epsilon_1)
\le
\phi(\Delta)-\frac{1}{\ell}\Delta+\epsilon_1.
\]
Thus, by \cref{main:R1_pre}, we obtain
\begin{equation}
\label{main:R1}
R_1\ll
HX^{\frac{1}{\ell}-1+\frac{1}{k}(\phi(\Delta)-\frac{1}{\ell}\Delta)+2\epsilon_1}
+
HX^{\frac{1}{k}+\frac{1}{\ell}-1}B^{-1}.
\end{equation}
This completes the estimate of $R_1$.

We next estimate the sum $R_2$. We use
\begin{equation}
\label{main:R2_int}
(X+H)^{\frac{\rho}{k}}-X^{\frac{\rho}{k}}
=
\frac{\rho}{k}\int_{X}^{X+H}u^{\frac{\rho}{k}-1}du
\ll
|\gamma|HX^{\frac{\beta}{k}-1}.
\end{equation}
Then, since $X/|\gamma|\ge1$ for $|\gamma|\le T\le X$,
\[
R_{2}
\ll
HX^{-1}\sum_{|\gamma|\le T}X^{\frac{\beta}{k}}
\ll
HX^{\frac{1}{\ell}-1}\sum_{|\gamma|\le T}\frac{X^{\frac{\beta}{k}}}{|\gamma|^{\frac{1}{\ell}}}.
\]
This right-hand side is the same quantity
appeared in \cref{main:R1_dyadic}.
Thus,
\begin{equation}
\label{main:R2}
R_{2}
\ll
HX^{\frac{1}{\ell}-1+\frac{1}{k}(\phi(\Delta)-\frac{1}{\ell}\Delta)+2\epsilon_1}
+
HX^{\frac{1}{k}+\frac{1}{\ell}-1}B^{-1}.
\end{equation}
This completes the estimate of $R_2$.

We finally estimate the sum $R_3$.
We dissect the sum dyadically to obtain
\begin{equation}
\label{main:R3_dyadic}
R_3
\ll
L\sup_{1\le K\le T}K^{-\frac{1}{2}}\sum_{K<|\gamma|\le2K}(HK)^{\frac{\beta}{k}}.
\end{equation}
We again write $K=X^{\frac{\delta}{k}}$
and use the parameter $\Delta$ defined in \cref{main:Delta}.
By \cref{main:T_range},
\[
X^{\frac{\Delta+\epsilon_1}{k}}=X^{1+\frac{\epsilon_1}{k}}H^{-1}=T\le X^{\frac{1}{k}}
\]
so that
\begin{equation}
\label{main:Delta_range}
0\le\Delta\le1-\epsilon_1.
\end{equation}
Let
\[
\lambda=\lambda(\delta)=\frac{\log K}{\log(HK)^{\frac{1}{k}}}
=\frac{k\log K}{\log H+\log K}
=\frac{\delta}{1-\frac{\Delta}{k}+\frac{\delta}{k}}.
\]
By \cref{main:Delta_range},
this function $\lambda(\delta)$ is increasing with respect to $\delta$.
Note that
\[
K
=K^{1-\frac{1}{k}}K^{\frac{1}{k}}
\le T^{1-\frac{1}{k}}K^{\frac{1}{k}}
\le (X^{1-\frac{1}{k}}K)^{\frac{1}{k}}
\le (HK)^{\frac{1}{k}}
\]
by \cref{main:T_range} provided \cref{HT_cond}.
Thus, by using \cref{lem:zero_sum} with $Y=(HK)^{\frac{1}{k}}$,
\[
K^{-\frac{1}{2}}\sum_{K<|\gamma|\le2K}(HK)^{\frac{\beta}{k}}
\ll
\left((HK)^{\frac{1}{k}(\phi(\lambda)-\frac{1}{2}\lambda)}
+(HK)^{\frac{1}{k}(1-\eta+(2\eta-\frac{1}{2})\lambda)}\right)L^A.
\]
Since
\[
HK=X(XH^{-1})^{-1}K=X^{1-\frac{\Delta}{k}+\frac{\delta}{k}},
\]
the last estimate is rewritten as
\begin{equation}
\label{main:R3_zero_sum}
\begin{aligned}
&K^{-\frac{1}{2}}\sum_{K<|\gamma|\le2K}(HK)^{\frac{\beta}{k}}\\
&\ll
\left(
X^{\frac{1}{k}(1-\frac{\Delta}{k}+\frac{\delta}{k})(\phi(\lambda)-\frac{1}{2}\lambda)}
+X^{\frac{1}{k}(1-\frac{\Delta}{k}+\frac{\delta}{k})(1-\eta+(2\eta-\frac{1}{2})\lambda)}
\right)L^A.
\end{aligned}
\end{equation}
Since both of the factors
\[
\left(1-\frac{\Delta}{k}+\frac{\delta}{k}\right),\quad
\left(\phi(\lambda)-\frac{1}{2}\lambda\right)
\]
are increasing function of $\delta$, by \cref{main:delta_range},
\begin{align}
X^{\frac{1}{k}(1-\frac{\Delta}{k}+\frac{\delta}{k})(\phi(\lambda)-\frac{1}{2}\lambda)}
&\le
X^{\frac{1}{k}(1+\epsilon_1)
(\phi(\lambda(\Delta+\epsilon_1))-\frac{1}{2}\lambda(\Delta+\epsilon_1))}\\
&\le
X^{\frac{1}{k}(\phi(\lambda(\Delta+\epsilon_1))-\frac{1}{2}\lambda(\Delta+\epsilon_1))
+\epsilon_1}.
\end{align}
Since
\[
\lambda'(\delta)
=
\frac{1-\frac{\Delta}{k}}{(1-\frac{\Delta}{k}+\frac{\delta}{k})^2}
\le
1-\frac{\Delta}{k}
\le
1\quad\text{for}\quad\Delta\le\delta\le\Delta+\epsilon_1,
\]
by \cref{lem:phi_dev} and the mean value theorem,
\begin{equation}
\phi(\lambda(\Delta+\epsilon_1))-\frac{1}{2}\lambda(\Delta+\epsilon_1)
\le
\phi(\lambda(\Delta))-\frac{1}{2}\lambda(\Delta)+\epsilon_1
=
\phi(\Delta)-\frac{1}{2}\Delta+\epsilon_1.
\end{equation}
Thus,
\begin{equation}
\label{main:R3_zero_sum1}
X^{\frac{1}{k}(1-\frac{\Delta}{k}+\frac{\delta}{k})(\phi(\lambda)-\frac{1}{2}\lambda)}
\le
X^{\frac{1}{k}(\phi(\Delta)-\frac{1}{2}\Delta)+2\epsilon_1}.
\end{equation}
Since
\begin{align}
\left(1-\frac{\Delta}{k}+\frac{\delta}{k}\right)
\left(1-\eta+\left(2\eta-\frac{1}{2}\right)\lambda\right)
&=
\left(1-\frac{\Delta}{k}+\frac{\delta}{k}\right)(1-\eta)+\left(2\eta-\frac{1}{2}\right)\delta\\
&=
\left(1-\frac{\Delta}{k}\right)(1-\eta)
+\left(\frac{1-\eta}{k}+2\eta-\frac{1}{2}\right)\delta,
\end{align}
we have
\begin{equation}
\label{main:R3_zero_sum2}
\begin{aligned}
X^{\frac{1}{k}(1-\frac{\Delta}{k}+\frac{\delta}{k})(1-\eta+(\eta-\frac{1}{2})\lambda)}
&\ll
\left\{
\begin{array}{ll}
H^{1-\eta}X^{(\frac{1}{2}+\eta)(\Delta+\epsilon_1)}&(\text{if $k=1$}),\\
H^{\frac{1-\eta}{k}}X^{\eta(\Delta+\epsilon_1)}&(\text{if $k\ge2$}),
\end{array}
\right.\\
&\ll
\left\{
\begin{array}{ll}
H^{\frac{1}{2}}X^{\frac{1}{2}+2\epsilon_1}&(\text{if $k=1$}),\\
H^{\frac{1}{k}}X^{2\epsilon_1}&(\text{if $k\ge2$}),
\end{array}
\right.
\end{aligned}
\end{equation}
for sufficiently large $X$.
By \cref{main:R3_dyadic}, \cref{main:R3_zero_sum},
\cref{main:R3_zero_sum1}, and \cref{main:R3_zero_sum2},
\begin{equation}
\label{main:R3}
R_3
\ll
\left\{
\begin{array}{ll}
X^{\frac{1}{k}(\phi(\Delta)-\frac{1}{2}\Delta)+3\epsilon_1}
+
H^{\frac{1}{2}}X^{\frac{1}{2}+3\epsilon_1}&(\text{if $k=1$}),\\
X^{\frac{1}{k}(\phi(\Delta)-\frac{1}{2}\Delta)+3\epsilon_1}
+
H^{\frac{1}{k}}X^{3\epsilon_1}&(\text{if $k\ge2$}).
\end{array}
\right.
\end{equation}

By combining \cref{main:decomp}, \cref{main:error}, \cref{main:main},
\cref{main:R1}, \cref{main:R2}, and \cref{main:R3},
we have
\[
\sum_{X<N\le X+H}R(N)
=
\GG{1}HX^{\frac{1}{k}+\frac{1}{\ell}-1}
+O\left(HX^{\frac{1}{k}+\frac{1}{\ell}-1}B^{-1}+E\right)
\]
provided
\begin{equation}
\label{main:H_cond_pre}
\begin{array}{ll}
X^{1-\min(\frac{1}{k},\frac{1}{2}(\frac{1}{k}+\frac{1}{\ell}),
\frac{2}{\ell})+\epsilon}\le H\le X^{1-\epsilon}
&(\text{if $k=1$}),\\
X^{1-\min(\frac{1}{k},\frac{1}{2}(\frac{1}{k}+\frac{1}{\ell}),
\frac{k}{\ell(k-1)})+\epsilon}\le H\le X^{1-\epsilon}
&(\text{if $k\ge2$}),
\end{array}
\end{equation}
and $\epsilon_1\le\frac{\epsilon}{16}$, where
\begin{equation}
\label{main:E}
E
=
HX^{\frac{1}{\ell}-1+\frac{1}{k}(\phi(\Delta)-\frac{1}{\ell}\Delta)+4\epsilon_1}
+
X^{\frac{1}{k}(\phi(\Delta)-\frac{1}{2}\Delta)+4\epsilon_1},\quad
XH^{-1}=X^{\frac{\Delta}{k}}.
\end{equation}
Let $\lambda_1,\lambda_2$ be the functions given by \cref{def:lambda},
or equivalently, given in \cref{lem:phi_solve}.
Then, By assuming further
\[
X^{1-\frac{\min(\lambda_1,\lambda_2)}{k}+\epsilon}\le H,
\]
we have $0\le\Delta\le\min(\lambda_1,\lambda_2)-k\epsilon$.
Thus, \cref{lem:phi_solve_shift} and \cref{main:E} implies
\[
E
\ll
HX^{\frac{1}{k}+\frac{1}{\ell}-1-\frac{\epsilon}{10}+4\epsilon_1}.
\]
Thus, by taking $\epsilon_1=\frac{\epsilon}{80}$,
we obtain the asymptotic formula \cref{LZ_asymp_kl} provided
\begin{equation}
\label{main:H_cond_prefinal}
\begin{array}{ll}
X^{1-\min(\frac{\lambda_1}{k},\frac{\lambda_2}{k},\frac{1}{k},\frac{1}{2}(\frac{1}{k}+\frac{1}{\ell}),
\frac{2}{\ell})+\epsilon}\le H\le X^{1-\epsilon}
&(\text{if $k=1$}),\\
X^{1-\min(\frac{\lambda_1}{k},\frac{\lambda_2}{k},\frac{1}{k},\frac{1}{2}(\frac{1}{k}+\frac{1}{\ell}),
\frac{k}{\ell(k-1)})+\epsilon}\le H\le X^{1-\epsilon}
&(\text{if $k\ge2$}).
\end{array}
\end{equation}

Our remaining task is to remove the exponents
$\frac{1}{k}$ and $\frac{1}{2}(\frac{1}{k}+\frac{1}{\ell})$ in \cref{main:H_cond_prefinal}.
Since $\lambda_1(\ell)\le1$ for any $\ell\ge2$, we have
\[
\frac{1}{k}\ge\frac{\lambda_1(\ell)}{k}.
\]
Thus, we can remove the exponent $\frac{1}{k}$ in \cref{main:H_cond_prefinal}.
Note that
\[
\frac{1}{2}\left(\frac{1}{k}+\frac{1}{\ell}\right)\ge\frac{\lambda_1(\ell)}{k}
\quad\Longleftrightarrow\quad
k\ge(2\lambda_1(\ell)-1)\ell=:\tilde{\lambda}_1(\ell).
\]
For the function $\tilde{\lambda}_1(\ell)$, we have
\[
\tilde{\lambda}_1(1),\ldots,\tilde{\lambda}_1(5)\le2,\quad
\tilde{\lambda}_1(6),\ldots,\tilde{\lambda}_1(9)\le1
\]
numerically and
\[
\ell\ge10\Longrightarrow
\tilde{\lambda}_1(\ell)\le(2\lambda_1(10)-1)\ell=0
\]
since $\lambda_1(\ell)$ is decreasing.
Thus,
\[
\frac{1}{2}\left(\frac{1}{k}+\frac{1}{\ell}\right)\ge\frac{\lambda_1(\ell)}{k}
\]
except the cases $(k,\ell)=(1,2),(1,3),(1,4),(1,5)$ for which we can check numerically
\[
\frac{1}{2}\left(\frac{1}{k}+\frac{1}{\ell}\right)\ge\frac{\lambda_2(k,\ell)}{k}.
\]
Thus, we can remove the exponent $\frac{1}{2}(\frac{1}{k}+\frac{1}{\ell})$ in \cref{main:H_cond_prefinal}.
This completes the proof.
\end{proof}

We next replace the exponent $\theta_C(k,\ell)$ by $\theta_A(k,\ell)$ as in \cref{thm:main_kl}.

\begin{lemma}
\label{lem:B_determined}
For positive integers $k,\ell$ with $\ell\ge2$, we have
\[
\frac{5}{12k}
\ge
\frac{k}{\ell(k-1)}
\]
if and only if
\begin{equation}
\label{B_determined:second}
\ell\ge10\quad\text{and}\quad
\frac{5}{24}\ell-\frac{1}{24}\sqrt{\ell(25\ell-240)}\le k\le\frac{5}{24}\ell+\frac{1}{24}\sqrt{\ell(25\ell-240)}.
\end{equation}
\end{lemma}
\begin{proof}
This lemma follows just by solving the quadratic inequality
\[
\frac{5}{12k}\ge\frac{k}{\ell(k-1)}
\quad\Longleftrightarrow\quad
\left(k-\frac{5}{24}\ell\right)^2\le\left(\frac{1}{24}\right)^2\ell(25\ell-240)
\]
for $k\ge2$. Note that \cref{B_determined:second} never holds for $k=1$ since
\begin{align}
\frac{5}{24}\ell-\frac{1}{24}\sqrt{\ell(25\ell-240)}
&=
\frac{5}{24}\ell-\frac{5}{24}\ell\sqrt{1-\frac{48}{5\ell}}\\
&>
\frac{5}{24}\ell-\frac{5}{24}\ell\left(1-\frac{24}{5\ell}\right)
=1
\end{align}
for $\ell\ge10$. This completes the proof.
\end{proof}

\begin{lemma}
\label{lem:comparisonAB}
Let $\theta_{A}(k,\ell),\theta_{B}(k,\ell)$ be functions given in \cref{thm:main_kl}.
Then, for positive integers $k,\ell$ with $\ell\ge2$, we have
\[
\theta_{B}(k,\ell)<\theta_{A}(k,\ell)
\quad\Longleftrightarrow\quad
\left\{
\begin{array}{c}
\ell\le9\ \text{and}\ \frac{5}{24}\ell<k,\\[3mm]
\text{or}\quad\ell\ge10\ \text{and}\ \frac{5}{24}\ell+\frac{1}{24}\sqrt{\ell(25\ell-240)}<k.
\end{array}
\right.
\]
\end{lemma}
\begin{proof}
We first consider the case $k\le\frac{5}{24}\ell$.
In this case,
\begin{align}
\theta_{A}(k,\ell)
&\le
\min\left(\frac{\lambda_2(k,\ell)}{k},\frac{k}{\ell(k-1)}\right)\\
&\le
\min\left(\frac{\lambda_2(k,\frac{24}{5}k)}{k},\frac{k}{\ell(k-1)}\right)\\
&=
\min\left(\frac{5}{12k},\frac{k}{\ell(k-1)}\right)
=
\theta_{B}(k,\ell)
\end{align}
by \cref{lem:monotonicity_lambda}.
Thus, in the case $k\le\frac{5}{24}\ell$, both hand sides of the assertion are false
so that the assertion holds.

We consider the remaining case $k>\frac{5}{24}\ell$.
In this case, by \cref{lem:B_determined},
\[
\theta_{B}(k,\ell)
=
\left\{
\begin{array}{>{\displaystyle}cl}
\frac{k}{\ell(k-1)}&
(\text{if $\ell\ge10$ and $\frac{5}{24}\ell<k%
\le\frac{5}{24}\ell+\frac{1}{24}\sqrt{\ell(25\ell-240)}$}),\\[4mm]
\frac{5}{12k}&(\text{otherwise}).
\end{array}
\right.
\]
Therefore, in the former case, i.e.~in the case
\begin{equation}
\label{comparison:case1}
\ell\ge10\quad\text{and}\quad\frac{5}{24}\ell<k\le \frac{5}{24}\ell+\frac{1}{24}\sqrt{\ell(25\ell-240)},
\end{equation}
we have
\[
\theta_{A}(k,\ell)
\le
\frac{k}{\ell(k-1)}
=
\theta_{B}(k,\ell).
\]
This again makes the both sides of the assertion false,
which proves the assertion for the case~\cref{comparison:case1}.

In the remaining case, in which \cref{comparison:case1} does not hold
but $k>\frac{5}{24}\ell$ holds,
we have
\begin{equation}
\frac{\lambda_1(\ell)}{k}>\frac{5}{12k},\quad
\frac{\lambda_2(k,\ell)}{k}>\frac{\lambda_2(k,\frac{24}{5}k)}{k}=\frac{5}{12k},\quad
\frac{k}{\ell(k-1)}>\frac{5}{12k}
\end{equation}
by \cref{lem:monotonicity_lambda} and \cref{lem:B_determined}.
Thus,
\[
\theta_{A}(k,\ell)
>
\frac{5}{12k}
=
\theta_{B}(k,\ell)
\]
in the remaining case. This completes the proof.
\end{proof}

\begin{lemma}
\label{lem:comparisonBC}
Let $\theta_{B}(k,\ell),\theta_{C}(k,\ell)$ be functions
given in \cref{lem:trivial} and \cref{lem:lemma_kl}, respectively.
Then, for positive integers $k,\ell$ with $\ell\ge2$, we have
\[
\theta_{B}(k,\ell)<\theta_{C}(k,\ell)
\quad\Longleftrightarrow\quad
\left\{
\begin{array}{c}
\ell\le9\ \text{and}\ \frac{5}{24}\ell<k,\\[3mm]
\text{or}\quad\ell\ge10\ \text{and}\ \frac{5}{24}\ell+\frac{1}{24}\sqrt{\ell(25\ell-240)}<k
\end{array}
\right.
\]
and
\[
\max(\theta_{B}(k,\ell),\theta_{C}(k,\ell))
=
\max(\theta_{A}(k,\ell),\theta_{B}(k,\ell)).
\]
\end{lemma}
\begin{proof}
If $k\ge2$, then this trivially holds
since $\theta_{A}(k,\ell)=\theta_{C}(k,\ell)$ for $k\ge2$.
Thus we consider the case $k=1$.
Since $\theta_{C}(1,\ell)\le\theta_{A}(1,\ell)$ for any case,
it suffices to prove that $\theta_{A}(1,\ell)\le\frac{2}{\ell}$
if $\theta_{B}(1,\ell)<\theta_{A}(1,\ell)$.
By \cref{lem:comparisonAB}, $\theta_{B}(1,\ell)<\theta_{A}(1,\ell)$ holds
if and only if $\ell=2,3,4$. For these cases, we have
\[
\theta_{A}(1,2)=\frac{17+4\sqrt{15}}{49}\le1,\quad
\theta_{A}(1,3)=\frac{44+24\sqrt{2}}{147}\le\frac{2}{3},\quad
\theta_{A}(1,4)=\frac{5}{11}\le\frac{1}{2}.
\] 
This completes the proof.
\end{proof}

We now complete the proof of main theorems.
Since \cref{thm:main_12} is just a special case of \cref{thm:main_kl},
we prove only \cref{thm:main_kl} and \cref{thm:main_k2}.

\begin{proof}[Proof of \cref{thm:main_kl}]
By \cref{lem:trivial} and \cref{lem:lemma_kl},
we have \cref{LZ_asymp_kl} provided
\[
X^{1-\max(\theta_{B}(k,\ell),\theta_{C}(k,\ell))+\epsilon}\le H\le X^{1-\epsilon}.
\]
Then the theorem follows by \cref{lem:comparisonBC}.
\end{proof}

\begin{proof}[Proof of \cref{thm:main_k2}]
By \cref{thm:main_kl} and \cref{lem:comparisonAB}, it suffices to prove
\begin{equation}
\label{main_k2:thetaA}
\theta_{A}(k,\ell)=\frac{1}{k}\quad\text{for $k\ge2$ and $\ell=2$}.
\end{equation}
Since $k\ge\ell$, we have
\[
\lambda_2(k,\ell)\ge1,\quad
\frac{k}{\ell(k-1)}\ge\frac{1}{\ell}\ge\frac{1}{k}.
\]
Also, $\lambda_1(2)=1$.
Thus, we obtain \cref{main_k2:thetaA}
and arrive at the theorem.
\end{proof}

\section{Comparison of the exponents}
\label{section:comparison}

In this section, we compare three exponents $\theta_{A},\theta_{B}$ and $\theta_{LZ}$.
As a preparation, we prove \cref{theta_determined},
which determines the value $\theta=\max(\theta_{A},\theta_{B})$ more precisely.
\begin{lemma}
\label{lem:theta_determined}
Let $\theta(k,\ell)$ be the function given in \cref{thm:main_kl}.
Then, for positive integers $k,\ell$ with $\ell\ge2$, we have
\[
\theta(k,\ell)
=
\left\{
\begin{array}{>{\displaystyle}cl}
\frac{\lambda_2(k,\ell)}{k}&(\text{for $(k,\ell)=(1,2),(1,3),(1,4)$}\\
&\hspace{8mm}\text{$(2,5),(2,6),(2,7),(2,8),(2,9)$}),\\[2mm]
\theta_{B}(k,\ell)&(\text{for $k=1$ and $\ell\ge5$}),\\[2mm]
\min\left(\frac{\lambda_1(\ell)}{k},\frac{k}{\ell(k-1)}\right)&(\text{otherwise}).
\end{array}
\right.
\]
\end{lemma}
\begin{proof}
We first consider the case $k=1$. In this case,
\cref{lem:comparisonAB} implies
\[
\theta(1,\ell)=\theta_{B}(1,\ell)\quad\text{for}\quad\ell\ge5.
\]
Some numerical computation tells us
\[
\theta(k,\ell)
=
\frac{\lambda_2(k,\ell)}{k}
\]
for the cases $(k,\ell)=(1,2),(1,3),(1,4)$.

We next consider the case $k>\ell$.
In this case, we have
\[
\frac{\lambda_2(k,\ell)}{k}
=
\frac{2}{3k}\left(\frac{k}{\ell}+\frac{1}{2}\right)
>
\frac{1}{k}
\ge
\frac{\lambda_1(\ell)}{k}.
\]
Therefore,
\begin{align}
\theta_{A}(k,\ell)
&=
\min\left(\frac{\lambda_1(\ell)}{k},\frac{\lambda_2(k,\ell)}{k},\frac{k}{\ell(k-1)}\right)\\
&=
\min\left(\frac{\lambda_1(\ell)}{k},\frac{k}{\ell(k-1)}\right)
\ge
\min\left(\frac{5}{12k},\frac{k}{\ell(k-1)}\right)
=
\theta_{B}(k,\ell)
\end{align}
so
\[
\theta(k,\ell)
=
\theta_{A}(k,\ell)
=
\min\left(\frac{\lambda_1(\ell)}{k},\frac{k}{\ell(k-1)}\right)
\]
as in the assertion.

We further consider the case $k\ge2$ and $\ell\ge22$.
If
\[
k\le\frac{5}{24}\ell+\frac{1}{24}\sqrt{\ell(25\ell-240)},
\]
then, since
\begin{align}
\frac{5}{24}\ell-\frac{1}{24}\sqrt{\ell(25\ell-240)}
&=
\frac{5}{24}\ell-\frac{5}{24}\ell\sqrt{1-\frac{48}{5\ell}}\\
&<
\frac{5}{24}\ell-\frac{5}{24}\ell\left(1-\frac{48}{5\ell}\right)
=2\le k,
\end{align}
\cref{lem:B_determined} implies
\[
\theta_{B}(k,\ell)
=\frac{k}{\ell(k-1)}
\ge\theta_{A}(k,\ell)
\quad\text{and}\quad
\frac{k}{\ell(k-1)}\le\frac{5}{12k}<\frac{\lambda_1(\ell)}{k}
\]
so that
\[
\theta(k,\ell)
=
\theta_{B}(k,\ell)
=
\frac{k}{\ell(k-1)}
=
\min\left(\frac{\lambda_1(\ell)}{k},\frac{k}{\ell(k-1)}\right).
\]
Therefore, for the case $k\ge2$ and $\ell\ge22$, it suffices to prove
\begin{equation}
\label{theta_determined:A}
\theta_{A}(k,\ell)
=
\min\left(\frac{\lambda_1(\ell)}{k},\frac{k}{\ell(k-1)}\right)
\end{equation}
provided
\begin{equation}
\label{theta_determined:finalcase}
k>\frac{5}{24}\ell+\frac{1}{24}\sqrt{\ell(25\ell-240)}
\end{equation}
since $\min\left(\frac{\lambda_1(\ell)}{k},\frac{k}{\ell(k-1)}\right)\ge\theta_{B}(k,\ell)$.
If $k>\frac{5}{8}\ell$ further holds, then
\[
\lambda_2(k,\ell)
=
\frac{2}{3}\left(\frac{k}{\ell}+\frac{1}{2}\right)
>
\frac{3}{4}
=
\lambda_1(3)
>
\lambda_1(\ell)
\]
so \cref{theta_determined:A} holds.
Thus we may assume $k\le\frac{5}{8}\ell$.
By \cref{theta_determined:finalcase}, we have
\begin{align}
k>\frac{5}{24}\ell+\frac{1}{24}\sqrt{\ell(25\ell-240)}
&=
\frac{5}{24}\ell+\frac{5}{24}\ell\sqrt{1-\frac{48}{5\ell}}\\
&>
\frac{5}{24}\ell+\frac{5}{24}\ell\left(1-\frac{48}{5\ell}\right)
=\frac{5}{12}\ell-2.
\end{align}
By using $\ell\ge22$, we further find that
\begin{equation}
\label{theta_determined:22kl}
\frac{5}{8}
\ge
\frac{k}{\ell}
>
\frac{5}{12}-\frac{1}{11}
=
\frac{43}{132}
>
\frac{31}{96}.
\end{equation}
Thus, by definition,
\[
\lambda_{2}(k,\ell)
=
\frac{10}{49}+\frac{2k}{7\ell}+\frac{4}{7}\sqrt{\frac{6}{7}\left(\frac{k}{\ell}-\frac{1}{7}\right)}.
\]
By using \cref{theta_determined:22kl} and \cref{lem:monotonicity_lambda}, we have
\[
\lambda_{2}(k,\ell)
\ge
\lambda_{2}\left(k,\frac{132}{43}k\right)
=
\frac{961+156\sqrt{22}}{3234}
>
\frac{1}{2}=\lambda_1(10)
>
\lambda_1(\ell)
\]
so \cref{theta_determined:A} holds.
Thus the assertion holds provided $k\ge2$ and $\ell\ge22$.

The remaining cases satisfy $2\le k\le\ell\le21$ so that only finitely many cases are remaining.
Therefore, we can use some numerical calculation to check that
the assertion holds even for the remaining cases. This completes the proof.
\end{proof}

We now prove that the case \cref{maincase} occur
if and only if \cref{comparison} holds.
Since the exponents $\theta_A$ and $\theta_B$ have been already compared
in \cref{lem:comparisonAB}, it suffices to prove the next lemma.
For completeness, we include the case $k=1$ as well.

\begin{lemma}
\label{lem:new_comparison}
Let $\theta_{LZ}(k,\ell),\theta(k,\ell)$ be functions
given in \cref{thm:LZ_kl} and \cref{thm:main_kl}, respectively.
Then, for positive integers $k,\ell$ with $\ell\ge2$, we have
\[
\theta_{LZ}(k,\ell)<\theta(k,\ell)
\quad\Longleftrightarrow\quad
\ell=2\quad\text{or}\quad
k<\lambda_1(\ell)\ell.
\]
\end{lemma}
\begin{proof}
In the case $\ell=2$, as we have seen in the proof of \cref{thm:main_k2},
\[
\theta(k,2)
=
\theta_{A}(k,2)
=
\frac{1}{k}>\frac{5}{6k}\ge\theta_{LZ}(k,2).
\]
Thus, the both sides of the assertion is true, so that the assertion itself is true.

We next consider the case $\ell\ge3$ and $k\ge\lambda_1(\ell)\ell$.
Since $\lambda_1(\ell)$ is decreasing,
\[
\theta_{A}(k,\ell)
\le\frac{\lambda_1(\ell)}{k}
=\min\left(\frac{\lambda_1(\ell)}{k},\frac{1}{\ell}\right)
\le\min\left(\frac{\lambda_1(3)}{k},\frac{1}{\ell}\right)
\le\theta_{LZ}(k,\ell).
\]
Again, since $\lambda_1(\ell)$ is decreasing,
the assumption $k\ge\lambda_1(\ell)\ell$ implies
\[
k\ge\lambda_1(\ell)\ell\ge\frac{5}{12}\ell
\]
so that
\[
\theta_{B}(k,\ell)
\le
\frac{5}{12k}
\le
\min\left(\frac{5}{6k},\frac{1}{\ell}\right)
=
\theta_{\mathrm{LZ}}(k,\ell).
\]
Combining two estimates above,
\[
\theta(k,\ell)
=
\max(\theta_{A}(k,\ell),\theta_{B}(k,\ell))
\le
\theta_{LZ}(k,\ell).
\]
Therefore, both sides of the assertion is false in the current case,
so that the assertion itself holds.

We finally consider the case $\ell\ge3$ and $k<\lambda_1(\ell)\ell$.
In this case, it suffices to prove
\begin{equation}
\label{new_comparison:final}
\theta(k,\ell)>\frac{1}{\ell}
\end{equation}
since $\theta_{LZ}(k,\ell)\le1/\ell$.
In the current case, we immediately have
\[
\frac{\lambda_1(\ell)}{k}>\frac{1}{\ell},\quad
\frac{k}{\ell(k-1)}>\frac{1}{\ell}.
\]
Therefore, by \cref{lem:theta_determined},
in order to prove \cref{new_comparison:final}, it suffices to consider the case
\begin{equation}
\label{new_comparison:exception1}
(k,\ell)
=
(1,3), (1,4), (2,5), (2,6), (2,7), (2,8), (2,9)
\end{equation}
and the case
\begin{equation}
\label{new_comparison:exception2}
k=1\quad\text{and}\quad\ell\ge5.
\end{equation}
In the case \cref{new_comparison:exception1}, we can check \eqref{new_comparison:final} numerically.
For the case \cref{new_comparison:exception2}, it suffices to see
\[
\frac{5}{12k}>\frac{1}{\ell}
\]
which trivially holds. This completes the proof.
\end{proof}

\subsection*{Acknowledgements}
The author would like to thank Prof.~Kohji Matsumoto
for his patient support and continuous encouragement.
The author also would like to thank Prof.~Hiroshi Mikawa
for his valuable advice and comments.
The author is indebted to
Prof.~Alessandro Languasco and Prof.~Alessandro Zaccagnini
for several stimulating communications and for helpful information.
This work was supported by Grant-in-Aid for JSPS Research Fellow
(Grant Number: JP16J00906).

\vspace{3mm}

\begin{flushleft}
{\small
{\sc
Graduate School of Mathematics, Nagoya University,\\
Chikusa-ku, Nagoya 464-8602, Japan.
}

{\it E-mail address}: {\tt m14021y@math.nagoya-u.ac.jp}
}
\end{flushleft}
\end{document}